\documentclass[a4paper,10pt]{article}
\usepackage[utf8]{inputenc}
\usepackage{amsthm}
\usepackage{amsfonts}
\usepackage{amsmath}
\usepackage{amssymb}
\usepackage{color}

\newtheorem{theorem}{Theorem}[section]
\newtheorem{lemma}[theorem]{Lemma}
\newtheorem{corollary}[theorem]{Corollary}
\newtheorem{proposition}[theorem]{Proposition}
\newtheorem{problem}[theorem]{Problem}

\title{Central polynomials with involution for the algebra of $2 \times 2$ upper triangular matrices}
\author{Ronald Ismael Quispe Urure \thanks{Supported by Ph.D. grant from CAPES}\\
Departamento de Matem\'atica, Universidade Federal de S\~ao Carlos\\
13565-905 S\~ao Carlos, SP, Brasil\\
e-mail: \texttt{urure6@gmail.com}\\ 
\\
\\
Dimas Jos\'e Gon\c{c}alves\thanks{Partially supported by FAPESP grant 
No. 2014/09310-5, and by CNPq grant No. 406401/2016-0}
\\
Departamento de Matem\'atica, Universidade Federal de S\~ao Carlos\\
13565-905 S\~ao Carlos, SP, Brasil\\
e-mail: \texttt{dimas@dm.ufscar.br}\\ 
\\
}

\begin{document}

\maketitle



\

\begin{abstract}
Let $F$ be a field of characteristic different from $2$, and let
$UT_2(F)$ be the algebra   
of $2 \times 2$ upper triangular matrices over $F$. For every involution of the first kind 
on $UT_2(F)$, 
we describe the set of all  
$*$-central polynomials for this algebra.
\end{abstract}

\section{Introduction}
Let $F$ be a field. In this paper, every algebra is unitary associative over $F$.  
Let  $F\langle X \rangle$ be the free unitary associative algebra, freely generated 
over $F$ by the infinite set $X=\{x_1,x_2,\ldots \}$.

A polynomial $f(x_1,\ldots,x_n) \in F\langle X \rangle$ is a 
polynomial identity for an algebra $A$ if 
\[f(a_1,\ldots,a_n)=0\]
for all $a_1,\ldots,a_n \in A$. 
Denote by $Id(A)$ the set of all polynomial identities for $A$. 
It is known that $Id(A)$ is a T-ideal, that is, an ideal 
closed under all
endomorphisms of $F\langle X \rangle$.

If $S\subseteq F\langle X \rangle$, we denote by $\langle S \rangle^T$ the 
T-ideal generated by $S$, that is, the intersection of
all T-ideals containing $S$. 
Given a T-ideal $I$, if there exists a finite set $S$ such that $I=\langle S \rangle^T$,
we say that $I$ is finitely
generated as a T-ideal. 

In 1950, Specht \cite{specht} posed the following problem:

\begin{problem}[Specht's problem] \label{problemspecht}  Is $Id(A)$ finitely generated,  as a T-ideal, 
for every algebra $A$ ?
\end{problem}

\noindent The answers to this question are: 

a)  Yes, if char$(F)=0$. Kemer \cite{kemer}.

b) No, if char$(F)\neq 0$. Belov (\cite{belov}), Grishin (\cite{grishin}) and Shchigolev (\cite{Shchigolev}).

\vspace{0.15cm}   

In general, the description of $Id(A)$ is a hard problem. 
The algebra $UT_n(F)$ of $n\times n$ upper triangular matrices 
plays an important role in the theory of PI-algebras.
Maltsev \cite{maltsev} described $Id(UT_n(F))$ when char$(F)=0$, and 
Siderov \cite{siderov} when $F$ is any field. 
In particular, they proved that $Id(UT_n(F))$ is finitely generated, as a T-ideal.

A T-space is a vector subspace of $F\langle X \rangle$ closed under all
endomorphisms of $F\langle X \rangle$. 
Every T-ideal is a T-space. Another important T-space is the set of all central polynomials for an 
algebra $A$, denoted by $C(A)$.
A polynomial $f(x_1,\ldots,x_n) \in F\langle X \rangle$ is a 
central polynomial for an algebra $A$ if 
\[f(a_1,\ldots,a_n)\in Z(A) \ \mbox{(center  of  $A$)}\]
for all $a_1,\ldots,a_n \in A$.
Note that 
\[C(A) \supseteq Id(A)+F.\]
Thus, some authors don't include $Id(A)+F$ in the definition of $C(A)$. In this paper, if $f \in Id(A)+F$, we say that
$f$ is a trivial central polynomial for $A$. 

Let $M_n(F)$ be the $n\times n$ matrix algebra. 
It is known that  
\[[x_1,x_2]^2 \]
is a non-trivial central polynomial for $M_2(F)$. Here, $[x_1,x_2]=x_1x_2-x_2x_1$ 
is the commutator of $x_1$ and $x_2$. 
In 1956, Kaplansky \cite{kaplansky} posed the following problem:

\begin{problem}
Does there exist a non-trivial central polynomial for $M_n(F)$ for all $n\geq 3$ ?  
\end{problem}

Formanek (\cite{formanek}) and Razmyslov (\cite{razmyslov}) answer ``yes'' to the question, and this was very 
important for ring theory.
Let $\tau(n)$ be the minimal degree of the non-trivial central polynomial for $M_n(F)$ when char$(F)=0$. 
We known that $\tau(1)=1$ and $\tau(2)=4$. Drensky and Kasparian \cite{drenskykasparia1,
drenskykasparia2} proved that  $\tau(3)=8$. We don't known $\tau(n)$ when
$n\geq 4$. It is an open problem.

If $S\subseteq F\langle X \rangle$, we denote by $\langle S \rangle^{TS}$ the 
T-space generated by $S$, that is, the intersection of
all T-spaces containing $S$. 
Given a T-space $I$, if there exists a finite set $S$ such that $I=\langle S \rangle^{TS}$,
we say that $I$ is finitely
generated as a T-space. Shchigolev \cite{shchigolev2} proved the following theorem:

\begin{theorem}
If $char(F)=0$ then every T-space is finitely generated. 
\end{theorem}

If $F$ is an infinite field of char$(F) > 2$, we have an important example of non-finitely generated T-space: it is $C(G)$, where $G$ is the 
infinite dimensional Grassmann algebra. See \cite{bekhocirrankin, brandaoplkrel}. 

It is well known that
\begin{equation} \label{cutnfigualidutnfmaisf}
C(UT_n(F))=Id(UT_n(F))+F
\end{equation}
for all $n\geq 2$. 
See   \cite[Exercise 1.4.2]{rowen} and  \cite[Example 3.2]{drenskyformanek}. 
For the algebra $M_2(F)$, the T-space $C(M_2(F))$ was described when $F$ is an infinite field of char$(F)\neq 2$. 
See \cite{colombo, okhitin}.

From now on, $F$ will be a field of char$(F) \neq 2$. Furthermore, we will consider algebras with involution 
of the first kind only.

Let $X=\{x_1,x_2,\ldots\}$, $X^*=\{x_1^*,x_2^*,\ldots\}$
be two disjoint infinite sets. Denote by $F\langle X \cup X^* \rangle$
the free unitary associative algebra, freely generated by $X \cup X^*$. 
Let $A$ be an algebra with involution $\circledast$.
A polynomial $f(x_1,x_1^*, \ldots,x_n,x_n^*) \in F\langle X \cup X^* \rangle$ is a polynomial identity with
involution (or $*$-polynomial identity) for $(A,\circledast)$ if 
\[f(a_1,a_1^{\circledast}, \ldots,a_n,a_n^{\circledast})=0\]
for all $a_1,\ldots,a_n \in A$. 
Denote by $Id(A, \circledast)$ the set of all $*$-polynomial identities for $(A, \circledast)$. 
It is known that $Id(A,\circledast)$ is a $T(*)$-ideal, that is, an $*$-ideal of $F\langle X\cup X^* \rangle$
closed under 
all $*$-endomorphisms of $F\langle X\cup X^* \rangle$. 

If $S\subseteq F\langle X\cup X^* \rangle$, we denote by $\langle S \rangle^{T(*)}$ the 
$T(*)$-ideal generated by $S$, that is, the intersection of
all $T(*)$-ideals containing $S$. 
Given a $T(*)$-ideal $I$, if there exists a finite set $S$ such that $I=\langle S \rangle^{T(*)}$,
we say that $I$ is finitely
generated as a $T(*)$-ideal. 

Recently, Aljadeff, Giambruno, Karasik (\cite{Aljgiakar})  and Sviridova (\cite{sviridova}) 
proved the following:

\begin{theorem}
 Let $F$ be a field of $char(F)=0$. If $A$ is an 
 algebra with involution $\circledast$, then
 $Id(A, \circledast)$ is finitely generated as a $T(*)$-ideal
\end{theorem}

Di Vincenzo, Koshlukov, La Scala \cite{vinkossca} described the involutions of the first kind on $UT_n(F)$. 
They proved that
there exist two classes of inequivalent involutions when $n$
is even and a single class otherwise. They also described: 

a) $Id(UT_2(F),\circledast)$ when $F$ is infinite,

b) $Id(UT_3(F),\circledast)$ when char$(F)=0$,

\noindent for all involutions of the first kind on $UT_2(F)$ and $UT_3(F)$ respectively. 
Urure and Gon\c{c}alves
\cite{ronalddimas}
described $Id(UT_2(F),\circledast)$ when $F$ is finite. 
In particular, $Id(UT_2(F),\circledast)$ is finitely generated as a $T(*)$-ideal (see \cite{ronalddimas,vinkossca}).
It is an open problem to describe $Id(UT_n(F),\circledast)$ in other cases.

Now, a $T(*)$-space is a vector subspace of $F\langle X \cup X^* \rangle$ closed under all
$*$-endomorphisms of $F\langle X\cup X^* \rangle$. 
Every $T(*)$-ideal is a $T(*)$-space. Another important $T(*)$-space is the set of all $*$-central polynomials for an 
algebra with involution $(A,\circledast)$, denoted by $C(A,\circledast)$.
A polynomial $f(x_1,x_1^*, \ldots,x_n,x_n^*) \in F\langle X \cup X^* \rangle$ is a central polynomial with
involution (or $*$-central polynomial) for $(A,\circledast)$ if 
\[f(a_1,a_1^{\circledast}, \ldots,a_n,a_n^{\circledast}) \in Z(A)\]
for all $a_1,\ldots,a_n \in A$.

If $W\subseteq F\langle X\cup X^* \rangle$, we denote by $\langle W \rangle^{TS(*)}$ the 
$T(*)$-space generated by $W$, that is, the intersection of
all $T(*)$-spaces containing $W$. 
Given a $T(*)$-space $I$, if there exists a finite set $W$ such that $I=\langle W \rangle^{TS(*)}$,
we say that $I$ is finitely
generated as a $T(*)$-space.

If $F$ is an infinite field, Brandão and Koshlukov \cite{brandaoplamen} decribed $C(M_2(F),\circledast)$
for every involution $\circledast$ on $M_2(F)$. Silva \cite{diogo} studied $C(M_{1,1}(G),\circledast)$.

In this paper, we describe $C(UT_2(F),\circledast)$ for every
involution of the first kind $\circledast$ and every field $F$ (finite or infinite) with char$(F) \neq 2$. In particular, we prove
that 
\[C(UT_2(F),\circledast)\neq Id(UT_2(F),\circledast)+F.\]
Compare this information with (\ref{cutnfigualidutnfmaisf}). 
Moreover, we prove the following theorem:

\begin{theorem}\label{teoremafinitogeradoresdeute}
Let $F$ be a field of $char(F) \neq 2$. If $\circledast$ is an involution of the first kind on $UT_2(F)$, then  
$C(UT_2(F),\circledast)$ is finitely generated as a $T(*)$-space.
\end{theorem}

\section{Involution}
From now on $F$ will be a field of char$(F)\neq 2$. Let $A$ be an unitary associative algebra over $F$.
A map $*: A \rightarrow A$ is an involution on $A$ if

a) $(a+b)^*=a^*+b^*$ for all $a,b \in A$,

b) $(ab)^*=b^*a^*$ for all $a,b \in A$,

c) $(a^*)^*=a$ for all $a \in A$.

\noindent Let $Z(A)$ be the center of $A$. If $a^*=a$ for all $a\in Z(A)$, then $*$
is called an involution of the first kind on $A$. Otherwise $*$ is called an involution of the second kind on $A$.
From now on we consider involutions of the first kind only. In this case $(\lambda a)^*=\lambda (a^*)$
for all $\lambda \in F$, $a\in A$. An element $a\in A$ is said to be symmetric if $a^*=a$. It's skew-symmetric if 
$a^*=-a$. Denote by $A^{+}$ and $A^{-}$ the following vector spaces: $A^{+}=\{a\in A: \ a^*=a\}$ and $A^{-}=\{a\in A: \ a^*=-a\}$. 
If $a\in A$ then 
\[a=(1/2)(a+a^*)+(1/2)(a-a^*).\]
Therefore
$A=A^{+} \oplus A^{-}$ as a vector space.

Let $(A,*)$ and $(B, \circ)$ be algebras with involutions $*$ and $\circ$ respectively.
We say that they are
isomorphic as algebras with involution if there exists an algebra isomorphism $\varphi: A \rightarrow B$ such that 
$\varphi (a^*)=(\varphi(a))^{\circ}$ for all $a\in A$. In this case we denote 
$(A,*) \simeq (B, \circ)$.

Denote by $\star$ and $s$ the following involutions on $UT_2(F)$:
\begin{equation}
  \left(
 \begin{array}{cc}
  a&c \\
  0&b
 \end{array}
\right)^{\star}=\left(
 \begin{array}{cc}
  b&c \\
  0&a
 \end{array}
\right)  \ \ \ \mbox{and} \ \ \ 
\left(
 \begin{array}{cc}
  a&c \\
  0&b
 \end{array}
\right)^{s}=\left(
 \begin{array}{cc}
  b&-c \\
  0&a
 \end{array}
\right)
\end{equation}
for all $a,b,c \in F$. By \cite[Propositions 2.5 and 2.6]{vinkossca} we have the next corollary: 

\begin{corollary}\label{corolarioequivalenciainvolucoe}
If $*$ is an involution of the first kind on $UT_2(F)$, then 
\[(UT_2(F),*)\simeq (UT_2(F),\star) \ \  or \ \ (UT_2(F),*)\simeq (UT_2(F),s).\]
Moreover, $(UT_2(F),\star)$ and $(UT_2(F),s)$ are not isomorphic as algebras with involution.
\end{corollary}

\section{$*$-polynomial identities and $*$-central polynomials}

Let $X=\{x_1,x_2,\ldots\}$ and $X^*=\{x_1^*,x_2^*,\ldots\}$
be two disjoint infinite sets. Denote by $F\langle X \cup X^* \rangle$
the free unitary associative algebra freely generated by $X \cup X^*$ over $F$. 
This algebra has an
involution $*:F\langle X \cup X^* \rangle \rightarrow F\langle X \cup X^* \rangle$ induced by the map 
$X\cup X^* \rightarrow  X \cup X^* $ defined by $x_i \rightarrow x_i^*$ and $x_i^* \rightarrow x_i$. For
example
\[(x_1x_2x_3^*+2x_1^*x_4)^*=x_3x_2^*x_1^*+2x_4^*x_1.\]
Let $F\langle Y \cup Z \rangle$ be the
free unitary associative algebra
freely generated by $Y\cup Z$ over $F$, where
\[Y=\{y_1,y_2,\ldots \}, \ Z=\{z_1,z_2,\ldots \}, \ y_i=x_i+x_i^*, \ z_i=x_i-x_i^*\]
for all $i\geq 1$. Note that $F\langle X \cup X^* \rangle=F\langle Y \cup Z \rangle$, $y_i$ is symmetric and 
$z_i$ is skew-symmetric. Thus $f(y_1,\ldots,y_n,z_1,\ldots,z_m)$ is a $*$-polynomial identity for 
an algebra with involution $(A,\circledast)$ if 
\[f(a_1,\ldots,a_n,b_1,\ldots,b_m)=0\]
for all $a_1,\ldots,a_n \in A^{+}$ and $b_1,\ldots,b_m\in A^{-}$. Denote by $Id(A,\circledast)$ 
the set of all $*$-polynomial
identities for $(A,\circledast)$.
This set is a $T(*)$-ideal that is an ideal invariant under all
$*$-endomorphisms of $F\langle X \cup X^* \rangle$. Here $*$-endomorphism means an endomorphism 
$\varphi$ of the algebra $F\langle X \cup X^* \rangle$ such that
\[\varphi (f^*)=(\varphi (f))^*\]
for all $f\in F\langle X \cup X^* \rangle$. 
In particular, if $f(y_1,\ldots,y_n,z_1,\ldots,z_m) \in Id(A,\circledast)$ then
\[f(g_1,\ldots,g_n,h_1,\ldots,h_m) \in Id(A,\circledast)\]
for all $g_1,\ldots,g_n \in F\langle Y \cup Z \rangle^{+}$ and $h_1,\ldots,h_m \in F\langle Y \cup Z \rangle^{-}$.

We denote by $\langle W \rangle^{T(*)}$ the $T(*)$-ideal of $F\langle Y \cup Z \rangle$ generated by $W$ that is
the smallest $T(*)$-ideal of $F\langle Y \cup Z \rangle$ containing $W$.

If $A$ is an algebra we denote the commutators as follows: 
\[[a_1,a_2]=a_1a_2-a_2a_1 \ \ \mbox{and} \ \ [a_1,\ldots,a_{n-1},a_n]=[[a_1,\ldots,a_{n-1}],a_n]\]
for all $a_1,\ldots,a_n \in A$, $n\geq 2$. A polynomial $f(y_1,\ldots,y_n,z_1,\ldots,z_m) \in F\langle Y\cup Z \rangle$ is 
called $Y$-proper if $f$ is a linear combination of polynomials
\[z_1^{r_1}\cdots z_m^{r_m}f_1 \cdots f_t\]
where $t \geq 0$, $r_1,\ldots,r_m \geq 0$ and $f_i \in F\langle Y \cup Z \rangle$ is a commutator of lenght $\geq 2$ for all $i=1,\ldots,t$.
Denote by $B$ the vector space of all $Y$-proper polynomials.
By the Poincar\'e-Birkhoff-Witt theorem every element 
$g(y_1,\ldots,y_n,z_1,\ldots,z_m) \in F\langle Y \cup Z \rangle$ is a linear combination of polynomials 
\begin{equation}\label{teoremapbw}
y_1^{s_1}\cdots y_n^{s_n}g_{(s_1,\ldots,s_n)}
\end{equation}
where $s_1,\ldots,s_n \geq 0$ and $g_{(s_1,\ldots,s_n)}\in B$.

Using \cite[Lemma 2.1]{drenskygiambruno} and similar arguments as in \cite[Proposition 4.3.11]{drenskybook} 
we state the following:

\begin{proposition}\label{baseparapropiosbaseparatudo}
Let $F$ be an infinite field of char$(F)\neq 2$. Let $I$ be a $T(*)$-ideal of $F\langle Y\cup Z \rangle$.
Consider $W \subset B$ such that
\[\{w+(B\cap I) \ : \ w\in W\}\]
is a basis for the quotient vector space $B/(B\cap I)$. Then the set of all polynomials 
\[
y_1^{s_1}\cdots y_n^{s_n}w+I,
\]
where $s_1,\ldots,s_n \geq 0,$ \ $n\geq 1$ and $w\in W$, is a basis for the quotient vector space $F\langle Y\cup Z \rangle/I$.
\end{proposition}

Recall that a polynomial $g(x_1,x_1^*, \ldots,x_n,x_n^*) \in F\langle X \cup X^* \rangle$ is 
a $*$-central polynomial for an algebra with involution $(A,\circledast)$ if 
\[g(a_1,a_1^{\circledast}, \ldots,a_n,a_n^{\circledast})\in Z(A)\]
for all $a_1,\ldots,a_n \in A$.  
Note that $f(y_1,\ldots,y_n,z_1,\ldots,z_m)\in F\langle Y\cup Z \rangle$ is a $*$-central polynomial 
for $(A,\circledast)$ if 
\[f(a_1,\ldots,a_n,b_1,\ldots,b_m)\in Z(A)\]
for all $a_1,\ldots,a_n \in A^{+}$ and $b_1,\ldots,b_m\in A^{-}$. Denote by $C(A,\circledast)$ 
the set of all $*$-central polynomials
for $(A,\circledast)$. This set is a $T(*)$-space that is a vector space invariant under all
$*$-endomorphisms of $F\langle X \cup X^* \rangle$. 

If $W \subseteq F\langle Y \cup Z \rangle$ then we denote by 
$\langle W \rangle^{TS(*)}$ the $T(*)$-space generated by $W$ that is
the smallest $T(*)$-space of $F\langle Y \cup Z \rangle$ containing $W$.
It's the vector space generated by the polynomials 
\[f(g_1,\ldots,g_n,h_1,\ldots,h_m) \]
where $f(y_1,\ldots,y_n,z_1,\ldots,z_m) \in W$ , \ 
$g_1,\ldots,g_n \in F\langle Y \cup Z \rangle^{+}$ and $h_1,\ldots,h_m \in F\langle Y \cup Z \rangle^{-}$.

Using similar arguments as in \cite[Proposition 4.2.3]{drenskybook} we state the following:

\begin{proposition}\label{propositconseq}
Let $F$ be a field (finite or infinite) with $|F|\geq q$. 
Let $f \in F\langle Y\cup Z \rangle$ and $w\in Y\cup Z$. Write
\[f=\sum_{i=0}^{d_w} f^{(i)} 
\]
where $f^{(i)}$ is the homogeneous component of $f$ with $\deg_{w} f^{(i)}=i$.
If $d_w < q$ then
\[\langle f \rangle^{TS(*)}=\langle f^{(0)},f^{(1)},\ldots,f^{(d_w)} \rangle^{TS(*)}.\]
\end{proposition}

Using Proposition \ref{propositconseq} and similar arguments as in 
\cite[Proposition 4.2.3]{drenskybook} we state the following:

\begin{proposition}\label{proposgeradoremultilemultih}
Let $I$ be a $T(*)$-space of $F\langle Y\cup Z \rangle$. 

\begin{itemize}
 \item[a)] If $F$ is an infinite field then $I$ is generated, as a $T(*)$-space, by its multihomogeneous elements.

 \item[b)] If $F$ is a field of char$(F)=0$ then $I$ is generated, as a $T(*)$-space, by its multilinear elements.

 \end{itemize}
 \end{proposition}

\begin{lemma}\label{partesimetricadefeidentidade}
  Let $f=f(y_1,\ldots,y_n,z_1,\ldots,z_m) \in F\langle Y\cup Z \rangle$ and write
 \[f=f^{+}+f^{-}\]
 where $f^{+}\in F\langle Y\cup Z \rangle^{+}$ and $f^{-}\in F\langle Y\cup Z \rangle^{-}$. Consider
 an algebra with involution $(A,\circledast)$. If 
 \[f(Y_1,\ldots,Y_n,Z_1,\ldots,Z_m)\in A^{-}\] 
 for all $Y_1,\ldots,Y_n \in A^{+}$ and $Z_1,\ldots,Z_m \in A^{-}$ then 
 $f^{+}\in Id(A,\circledast)$.
 \end{lemma}

\begin{proof}
Let $Y_1,\ldots,Y_n \in A^{+}$ and $Z_1,\ldots,Z_m \in A^{-}$. 
Since $f^{+}\in F\langle Y\cup Z \rangle^{+}$ we have $f^{+}(Y_1,\ldots,Y_n,Z_1,\ldots,Z_m)\in A^{+}$.
Since 
\[f(Y_1,\ldots,Y_n,Z_1,\ldots,Z_m)-f^{-}(Y_1,\ldots,Y_n,Z_1,\ldots,Z_m)=
 f^{+}(Y_1,\ldots,Y_n,Z_1,\ldots,Z_m)\]
 we have $f^{+}(Y_1,\ldots,Y_n,Z_1,\ldots,Z_m)\in A^{-}$ too. Therefore 
 \[f^{+}(Y_1,\ldots,Y_n,Z_1,\ldots,Z_m)=0\]
 as desired.
 \end{proof}

\section{$*$-central polynomials for $(UT_2(F),\star)$}

In this section, we describe the $*$-central polynomials for $(UT_2(F),\star)$ where  
\begin{equation}\label{definicaoprimeirainvolucao}
  \left(
 \begin{array}{cc}
  a&c \\
  0&b
 \end{array}
\right)^{\star}=\left(
 \begin{array}{cc}
  b&c \\
  0&a
 \end{array}
\right) 
\end{equation}
for all $a,b,c \in F$. Note that $\{e_{11}+e_{22}, \ e_{12}\}$ and $\{e_{11}-e_{22}\}$ form a basis for the
vector spaces $(UT_2(F))^{+}$ and $(UT_2(F))^{-}$ respectively.

The next lemma is proved in \cite[Theorem 3.1]{vinkossca}. See \cite[Lemma 5.2]{ronalddimas} too.

\begin{lemma}\label{lemavariavcomutam}
 Let $F$ be a field of $char(F)\neq 2$. If $I=Id(UT_2(F),\star)$ then
\begin{eqnarray*}
z_{\sigma (1)} \cdots z_{\sigma (m)}[z_{\sigma (m+1)},y_1]+I&=&z_{1} \cdots z_{m}[z_{m+1},y_1]+I \ \ and \\
z_{1} \cdots z_{m}[z_{m+1},y_1]+I&=&(-1)^{m-i}z_{1} \cdots z_{i} [z_{m+1},y_1] z_{i+1} \cdots  z_{m}+I
\end{eqnarray*}
for all $\sigma \in S_{m+1}$. 
\end{lemma}

\begin{lemma}\label{polcentral1} 
 Let $F$ be a field of $char(F)\neq 2$. Let  
 \[f(z_1,\ldots,z_m)=z_1 z_2 \cdots z_{m}\]
 where $m\geq 0$. If $m$ is even then $f \in C(UT_2(F),\star)$. If $m$ is odd then $f\notin C(UT_2(F),\star)$.
\end{lemma}
\begin{proof}
Suppose $m=2n$ for some $n\geq 0$. If $Z_i=\lambda_i (e_{11}-e_{22})$, where $\lambda_i \in F$, then
\[f(Z_1,\ldots,Z_{2n})=(\lambda_1\cdots \lambda_{2n})(e_{11}+e_{22}) \in Z(UT_2(F)).\]
Thus $f \in  C(UT_2(F),\star)$. 

Suppose $m=2n+1$ for some $n\geq 0$. Then 
\[f(e_{11}-e_{22},e_{11}-e_{22},\ldots,e_{11}-e_{22})=e_{11}-e_{22}.\]
Therefore $f\notin C(UT_2(F),\star)$.
\end{proof}

\begin{proposition}\label{propositioncomznadireita}
Let $F$ be a field of $char(F)\neq 2$.
Let 
\[g(y_1,\ldots,y_n,z_1,\ldots,z_m)=f(y_1,\ldots,y_n,z_1,\ldots,z_m)z_{m}\] be a polynomial where 
$f$ is a homogeneous polynomial in the variable $z_m$. If $g \in C(UT_2(F), \star)$ then 
\[g \in \left( Id(UT_2(F), \star)+ 
\langle z_1z_2 \rangle^{TS(*)}\right) . 
\]
\end{proposition}

\begin{proof}
Let $Y_1,\ldots,Y_n \in UT_2(F)^{+}$ and $Z_1,\ldots,Z_m \in UT_2(F)^{-}$. Then
\[f(Y_1,\ldots,Y_n,Z_1,\ldots,Z_m)=
   \begin{pmatrix}
a & c \\ 
0 & b
\end{pmatrix} \ \ \mbox{and} \ \ Z_{m}=\left(\begin{array}{cc}
d & 0 \\ 
0 & -d
\end{array}\right)\]
for some $a,b,c,d \in F$. Since $f(y_1,\ldots,y_n,z_1,\ldots,z_m)z_{m}\in C(UT_2(F), \star)$, we obtain
\[f(Y_1,\ldots,Y_n,Z_1,\ldots,Z_m)Z_{m}=
   \begin{pmatrix}
ad & -cd \\ 
0 & -bd
\end{pmatrix} \in Z(UT_2(F)).\]
Thus $cd=0$ and $ad=-bd$. 

We have two cases:

\

Case 1. $\deg_{z_m}f=0$.

In this case, 
$f(Y_1,\ldots,Y_n,Z_1,\ldots,Z_m)=f(Y_1,\ldots,Y_n,Z_1,\ldots,Z_{m-1})$. 
If $d=1$ then $a=-b$ and $c=0$. Thus $f(Y_1,\ldots,Y_n,Z_1,\ldots,Z_m)\in UT_2(F)^{-}$.

\

Case 2. $\deg_{z_m}f\geq 1$.

In this case, if $d=0$ then $f(Y_1,\ldots,Y_n,Z_1,\ldots,Z_m)=0 \in UT_2(F)^{-}$. If $d\neq 0$ 
then $a=-b$ and $c=0$. Thus $f(Y_1,\ldots,Y_n,Z_1,\ldots,Z_m)\in UT_2(F)^{-}$.

\

By the two cases we have $f(Y_1,\ldots,Y_n,Z_1,\ldots,Z_m)\in UT_2(F)^{-}$ for all
$Y_1,\ldots,Y_n \in UT_2(F)^{+}$ and $Z_1,\ldots,Z_m \in UT_2(F)^{-}$. By Lemma 
\ref{partesimetricadefeidentidade} we can write $f=f^{+}+f^{-}$ where $f^{+} \in Id(UT_2(F),\star)$ and
$f^{-} \in F\langle Y\cup Z \rangle^{-}$. Thus
\[fz_{m}=f^{+}z_{m}+f^{-}z_{m}\in \left( Id(UT_2(F), \star)+ 
\langle z_1z_2 \rangle^{TS(*)}\right).
\]
The proof is complete. 
\end{proof}

\begin{lemma}\label{coromatrizantisimetrica2}
Let $F$ be a field of $char(F)\neq 2$.
If $n\geq 1$ then 
\[z_1 \cdots	z_{2n}\in \left( Id(UT_2(F), \star)+ 
\langle z_1z_2 \rangle^{TS(*)}\right) . 
\]
\end{lemma}
\begin{proof}
This is a direct consequence of Lemma \ref{polcentral1} and  Lemma \ref{propositioncomznadireita}.
\end{proof}

\subsection{$C(UT_2(F),\star)$ when char$(F)=0$}

The next theorem was proved in \cite{vinkossca}. See Proposition \ref{baseparapropiosbaseparatudo} and 
\cite[Theorem 3.1]{vinkossca} for details. 
 
\begin{theorem}  \label{teoremaidentidaprimeirainv}
Let $F$ be an infinite field of char$(F)\neq 2$. Consider the involution $\star$ defined
 in (\ref{definicaoprimeirainvolucao}). Denote by $I$ the $T(*)$-ideal generated by the polynomials
 \begin{eqnarray*}
[y_1,y_2], \ \ [z_1,z_2], \ \  [y_1,z_1][y_2,z_2] \ \ \mbox{and} \ \  z_1y_1z_2-z_2y_1z_1.
\end{eqnarray*}
Then $Id(UT_2(F),\star)=I$. Moreover, the quotient vector space $F\langle Y\cup Z \rangle /I$ 
 has a basis consisting of all polynomials of the form
 \begin{equation}\label{geradoresquocientepori}
 y_1^{s_1}\cdots y_n^{s_n} z_1^{r_1}\cdots z_m^{r_m}[z_m,y_k]+I \ \ \ \mbox{and} \ \  \
  y_1^{s_1}\cdots y_n^{s_n} z_1^{r_1}\cdots z_m^{r_m}+I
  \end{equation}
  where $n\geq 1$, \  $m\geq 1$, \  $s_1,\ldots, s_n, r_1,\ldots,r_m\geq 0$, \  $k\geq 1$.
\end{theorem}

Now we will prove the first main theorem of this paper.

\begin{theorem} \label{teorinvsta1}
 Let $F$ be a field of char$(F)=0$. Consider the involution $\star$ defined
 in (\ref{definicaoprimeirainvolucao}). The set of all $*$-central polynomials of $(UT_2(F), \star)$ is
 \[C(UT_2(F),\star)= Id(UT_2(F), \star)+ 
\langle z_1z_2 \rangle^{TS(*)}+F.\]
\end{theorem}

\begin{proof}
Denote $I=Id(UT_2(F), \star)$ and $C=C(UT_2(F),\star)$.
By Lemma \ref{polcentral1}, we have 
\[C \supseteq (I+ 
\langle z_1z_2 \rangle^{TS(*)}+F). \]
Let $f(y_1,\ldots,y_n,z_1,\ldots,z_m) \in C$ be a multilinear polynomial.
We shall prove that $f \in (I+\langle z_1z_2 \rangle^{TS(*)}+F)$. 
By Theorem \ref{teoremaidentidaprimeirainv},
we have $f+I=\overline{f}+I$ where
\[\overline{f}=\alpha y_1\cdots y_n z_1\cdots z_m + 
\sum_{k=1}^n \alpha_k y_1\cdots \widehat{y_k} \cdots y_n z_1\cdots z_{m-1}[z_m,y_k],\]
for some $\alpha,  \alpha_k \in F$. Thus, there exists $g\in I$ such that $f=\overline{f}+g$. In particular,
$\overline{f}=(f-g) \in C$.

\

Case 1. $n=0$ and $m=0$.

In this case, $\overline{f}=\alpha$ and so $f\in (F+I) \subset (I+\langle z_1z_2 \rangle^{TS(*)}+F)$.

\

Case 2. $n=0$ and $m>0$. 

In this case, 
\[\overline{f}=\alpha z_1\cdots z_m.\]
By Lemma \ref{polcentral1}, we have that $\alpha =0$ or $m$ is even. By Lemma
\ref{coromatrizantisimetrica2}, we obtain $f\in (I+\langle z_1z_2 \rangle^{TS(*)}+F)$.

\

Case 3. $n>0$ and $m=0$.

In this case, 
\[\overline{f}=\alpha y_1\cdots y_n.\]
If $\alpha \neq 0$ then
\[\overline{f}(1,\ldots,1, e_{12})=\alpha e_{12}.\]
Thus $\overline{f} \notin C$, which is a contradiction. Therefore $\alpha=0$ and 
$f\in I \subset (I+\langle z_1z_2 \rangle^{TS(*)}+F)$.

\

Case 4. $n>0$ and $m=1$. 

In this case, 
\[\overline{f}=\alpha y_1\cdots y_n z_1 + 
\sum_{k=1}^n \alpha_k y_1\cdots \widehat{y_k} \cdots y_n [z_1,y_k].\]
Since $\overline{f}\in C$ we have $\overline{f}(1,\ldots,1,e_{11}-e_{22})=\alpha (e_{11}-e_{22}) \in 
Z(UT_2(F))$. Thus $\alpha =0$ and 
\[\overline{f}=\sum_{k=1}^n \alpha_k y_1\cdots \widehat{y_k} \cdots y_n [z_1,y_k].\]
Since $\overline{f}\in C$ we obtain $\overline{f}(1,\ldots ,1,e_{12},1,\ldots ,1, e_{11}-e_{22})=2\alpha_k e_{12}
\in 
Z(UT_2(F))$. Thus $\alpha_k=0$ for all $k=1,\ldots,n$. We prove that 
$f\in I \subset (I+\langle z_1z_2 \rangle^{TS(*)}+F)$.

\

Case 5. $n>0$ and $m\geq 2$.

By Lemma \ref{lemavariavcomutam}, we have $f+I=\overline{f}+I=\overline{\overline{f}}+I$ where
\[\overline{\overline{f}}=\alpha y_1\cdots y_n z_1\cdots z_m - 
\sum_{k=1}^n \alpha_k y_1\cdots \widehat{y_k} \cdots y_n z_1\cdots z_{m-2}[z_{m-1},y_k]z_m.\]
Since $I\subset C$ we have $\overline{\overline{f}} \in C$. By Proposition \ref{propositioncomznadireita},
we obtain $\overline{\overline{f}} \in (I+\langle z_1z_2 \rangle^{TS(*)})$. Thus 
$f \in (I+\langle z_1z_2 \rangle^{TS(*)}) \subset (I+\langle z_1z_2 \rangle^{TS(*)}+F)$.

By the five cases and by Proposition \ref{proposgeradoremultilemultih} we have 
$C=I+\langle z_1z_2 \rangle^{TS(*)}+F$ as desired.
\end{proof}

\subsection{$C(UT_2(F),\star)$ when $F$ is an infinite field of char$(F)>2$}

We start this section with the next proposition. 
Similar result is obtained in \cite[Theorem 6 in Chapter 4]{bahturinbook} when we consider $T$-ideals of the free 
Lie algebra. Moreover, similar result is obtained when we consider $T$-spaces of the free associative algebra.

\begin{proposition}\label{proposicaocorpoinfinitopotenciasdep}
 Let $F$ be an infinite field of $char(F)=p >2$. If $H$ is a $T(*)$-space then $H$ is generated, as a $T(*)$-space,
 by its multihomogeneous elements $f(y_1,\ldots,y_n,z_1,\ldots,z_m)\in H$ with multidegree 
 $(p^{a_1},\ldots,p^{a_n},p^{b_1},\ldots,p^{b_m})$ where $a_1,\ldots,a_n,b_1,\ldots,b_m \geq 0$.
\end{proposition}

\begin{proof}
Denote by $H_M$ the set of all multihomogeneous elements of $H$, and by $H_{PM}$ the set of all
multihomogeneous elements $f(y_1,\ldots,y_n,z_1,\ldots,z_m)\in H$ with multidegree 
 $(p^{a_1},\ldots,p^{a_n},p^{b_1},\ldots,p^{b_m})$ where $a_1,\ldots,a_n,b_1,\ldots,b_m \geq 0$,
 $n\geq 0$ and $m\geq 0$.
 
By Proposition \ref{proposgeradoremultilemultih} it follows that \[H=\langle H_M \rangle ^{TS(*)}.\]
We have to prove $\langle H_M \rangle ^{TS(*)}=\langle H_{PM}\rangle^{TS(*)}$. It's clear that $\langle H_M \rangle ^{TS(*)} \supseteq \langle H_{PM}\rangle^{TS(*)}$.
Note that 
\begin{equation}
\langle H_M \rangle ^{TS(*)} \subseteq \langle H_{PM}\rangle^{TS(*)} \Leftrightarrow H_M \subseteq \langle H_{PM}\rangle^{TS(*)}.
 \end{equation}

Let $g(y_1,\ldots,y_n,z_1,\ldots,z_m) \in H_M$. 

a) If $g\in H_{PM}$ then $g\in \langle H_{PM}\rangle^{TS(*)}$.

b) Suppose $g \notin H_{PM}$. Denote by 
\[d=(d_{y_1},\ldots, d_{y_n},d_{z_1}, \ldots, d_{z_m})\]
the  multidegree of $g$.
Without loss of generality, we may assume that $\deg_{y_1}g=d_{y_1}$ is not a power of $p$.
Let $\deg_{y_1}g=p^kq$ where $(p,q)=1$. Denote by $\overline{g}(y_1,y_2,\ldots,y_n,y_{n+1},z_1,\ldots,z_m)$ the 
multihomogeneous component of 
\[g(y_1+y_{n+1},y_2,\ldots , y_n, z_1, \ldots,z_m)\]
with multidegree 
\[\overline{d}=(\overline{d}_{y_1},\overline{d}_{y_2},\ldots, \overline{d}_{y_n},\overline{d}_{y_{n+1}},\overline{d}_{z_1}, \ldots, \overline{d}_{z_m})=
(\overline{d}_{y_1},d_{y_2},\ldots, d_{y_n},\overline{d}_{y_{n+1}},d_{z_1}, \ldots, d_{z_m})\]
where
\[\deg_{y_1}\overline{g}=\overline{d}_{y_1}=p^k \ \ \mbox{and} \ \  \deg_{y_{n+1}}\overline{g}=\overline{d}_{y_{n+1}}=p^kq-p^k.\]
Since $F$ is an infinite field, we have $\overline{g} \in \langle g \rangle^{TS(*)}$. It is known that
\[ {p^kq \choose p^k}=q \neq 0 \mod p. \]
Thus, since 
\[\overline{g}(y_1, y_2, \ldots,y_n,y_{1},z_1,\ldots,z_m)={p^kq \choose p^k} g(y_1, \ldots,y_n,z_1,\ldots,z_m),\]
we have $g\in \langle \overline{g} \rangle^{TS(*)}$. We prove that $\langle g \rangle^{TS(*)}=\langle \overline{g} \rangle^{TS(*)}$.
Now we can use the same arguments in $\overline{g}$. After a few steps, we will obtain
$\langle g \rangle^{TS(*)}=\langle h \rangle^{TS(*)}$ for some $h\in H_{PM}$. Thus $g\in \langle H_{PM}\rangle^{TS(*)}$.
We prove that $\langle H_M \rangle ^{TS(*)} \subseteq \langle H_{PM}\rangle^{TS(*)}$ as desired.
\end{proof}

\begin{lemma} \label{lemacorpoinfinitopotdep}
Let $F$ be an infinite field of char$(F)=p >2$. Let $L$ be the $T(*)$-space
\[L= Id(UT_2(F), \star)+ 
\langle z_1z_2 \rangle^{TS(*)}+  \langle  y_1^p \rangle^{TS(*)}.\]

a) Then $L \subseteq C(UT_2(F),\star)$. 

b) Let $a_1,\ldots,a_n \geq 1$. 
If $f(y_1,\ldots,y_n)$ is a multihomogeneous polynomial with multidegree $(p^{a_1},\ldots,p^{a_n})$ 
 then $f \in L$.
\end{lemma}

\begin{proof}
a) Let $Y\in UT_2(F)^{+}$. Thus $Y=\begin{pmatrix}
a & b \\ 
0 & a
\end{pmatrix}$ for some  $a,b \in F$. If $i\geq 1$ then
\[Y^i=
\left( \begin{array}{cc}
  a^i& ia^{i-1}b \\
  0 & a^i
 \end{array}\right).
\]
Therefore $Y^p \in Z(UT_2(F))$ and $y_1^p \in C(UT_2(F),\star)$. 
By Lemma \ref{polcentral1}, we have
$z_1z_2 \in C(UT_2(F),\star)$. 
Therefore $L \subseteq C(UT_2(F),\star)$.

b) Denote $I=Id(UT_2(F), \star)$. By Theorem \ref{teoremaidentidaprimeirainv},
\[f+I=\alpha y_1^{p^{a_1}} y_2^{p^{a_2}} \cdots y_n^{p^{a_n}}+I\]
for some $\alpha \in F$. Since $[y_i,y_j] \in I$ (see Theorem \ref{teoremaidentidaprimeirainv}),
we obtain
\begin{eqnarray*}
f+I&=&\alpha y_1^{p^{a_1}}  \cdots y_n^{p^{a_n}}+I=
\alpha \left(y_1^{p^{a_1-1}}  \cdots y_n^{p^{a_n-1}}\right)^p +I\\
&=&\alpha 
\left(1/2( y_1^{p^{a_1-1}}  \cdots y_n^{p^{a_n-1}}+y_n^{p^{a_n-1}}  \cdots y_1^{p^{a_1-1}}) \right)^p +I\\
&=&\alpha g^p +I,
\end{eqnarray*}
where $g=1/2( y_1^{p^{a_1-1}}  \cdots y_n^{p^{a_n-1}}+y_n^{p^{a_n-1}}  \cdots y_1^{p^{a_1-1}})$. Since $g$
is a symmetric polynomial, we have $\alpha g^p \in \langle  y_1^p \rangle^{TS(*)}$ and therefore
$f \in (I+\langle  y_1^p \rangle^{TS(*)}) \subseteq L$. 
The proof is complete.
\end{proof}

\begin{theorem}\label{teorinvsta2}
Let $F$ be an infinite field of char$(F)=p >2$. Consider the involution $\star$ defined
in (\ref{definicaoprimeirainvolucao}). The set of all $*$-central polynomials of $(UT_2(F), \star)$ is
\[C(UT_2(F),\star)= Id(UT_2(F), \star)+ 
\langle z_1z_2 \rangle^{TS(*)}+  \langle  y_1^p \rangle^{TS(*)}.\]
\end{theorem}

\begin{proof}

Denote $I=Id(UT_2(F), \star)$ and $C=C(UT_2(F),\star)$.
By Lemma \ref{lemacorpoinfinitopotdep}, we have 
\[C \supseteq (I+ 
\langle z_1z_2 \rangle^{TS(*)}+\langle  y_1^p \rangle^{TS(*)}). \]
Let $f(y_1,\ldots,y_n,z_1,\ldots,z_m) \in C$ be a multihomogeneous polynomial with multidegree
$(p^{a_1},\ldots,p^{a_n},p^{b_1},\ldots,p^{b_m})$ 
where $a_1,\ldots,a_n, b_1, \ldots, b_m \geq 0$. 
We shall prove that $f \in (I+\langle z_1z_2 \rangle^{TS(*)}+\langle  y_1^p \rangle^{TS(*)})$. 
By Theorem \ref{teoremaidentidaprimeirainv}
we obtain $f+I=\overline{f}+I$ where
\begin{eqnarray*}
\overline{f}&=&
 \sum_{k=1}^n \alpha_k y_1^{p^{a_1}}\cdots y_k^{p^{a_k}-1} \cdots  y_n^{p^{a_n}} z_1^{p^{b_1}}\cdots z_m^{p^{b_m}-1}[z_m,y_k]+\\
 &&+\alpha y_1^{p^{a_1}} \cdots  y_n^{p^{a_n}} z_1^{p^{b_1}}\cdots z_m^{p^{b_m}}
\end{eqnarray*}
for some $\alpha_1,\ldots,\alpha_n, \alpha \in F$.

\

Case 1. $n=0$ and $m=0$.

In this case, $f=\alpha$ and so 
\[f\in F \subset 
 \langle  y_1^p \rangle^{TS(*)} \subset (I+ 
\langle z_1z_2 \rangle^{TS(*)}+\langle  y_1^p \rangle^{TS(*)}).\]

\

Case 2. $n\geq 1$ and $m=0$.

In this case, 
\[\overline{f}=\alpha y_1^{p^{a_1}} \cdots  y_n^{p^{a_n}}.\]

If $a_i=0$, for some $i$, then 
\[\overline{f}=\alpha y_1^{p^{a_1}} \cdots y_i \cdots  y_n^{p^{a_n}}\]
and 
$\overline{f}(1,\ldots,1,y_i,1,\ldots,1)= \alpha y_i \in C$. Thus $\alpha =0$, $\overline{f}=0$ and 
$f\in I \subset (I+ 
\langle z_1z_2 \rangle^{TS(*)}+\langle  y_1^p \rangle^{TS(*)})$.

Suppose $a_1,\ldots,a_n \geq 1$. By Lemma \ref{lemacorpoinfinitopotdep},  
\[\overline{f} \in (I+ 
\langle z_1z_2 \rangle^{TS(*)}+\langle  y_1^p \rangle^{TS(*)})\] and therefore
$f \in (I+ 
\langle z_1z_2 \rangle^{TS(*)}+\langle  y_1^p \rangle^{TS(*)})$.

\

Case 3. $m=1$ and $b_m=0$.

In this case, 
\[\overline{f}=
 \sum_{k=1}^n \alpha_k y_1^{p^{a_1}}\cdots y_k^{p^{a_k}-1} \cdots  y_n^{p^{a_n}}[z_1,y_k]+
 \alpha y_1^{p^{a_1}} \cdots  y_n^{p^{a_n}} z_1.
\]

Since $\overline{f} \in C$ we have $\overline{f}(1,\ldots,1,z_1)=\alpha z_1 \in C$. Thus $\alpha =0$ and 
\[\overline{f}=
 \sum_{k=1}^n \alpha_k y_1^{p^{a_1}}\cdots y_k^{p^{a_k}-1} \cdots  y_n^{p^{a_n}}[z_1,y_k].\]

 If $Y_1=\ldots = Y_{k-1}= Y_{k+1}=\ldots = Y_n=e_{11}+e_{22}$, $Y_k=e_{11}+e_{22}+e_{12}$ and 
$Z_1=e_{11}-e_{22}$, then
\[\overline{f}(Y_1,\ldots,Y_n,Z_1)=  2 \alpha_k e_{12} \in Z(UT_2(F)).\]
Thus $\alpha_k=0$ for all $k=1,\ldots,n$ and $\overline{f}=0$. Therefore
\[f \in I \subseteq (I+ 
\langle z_1z_2 \rangle^{TS(*)}+\langle  y_1^p \rangle^{TS(*)}).\]

\

Case 4. $m\geq 2$ and $b_m=0$.

In this case,
\begin{eqnarray*}
\overline{f}&=&
 \sum_{k=1}^n \alpha_k y_1^{p^{a_1}}\cdots y_k^{p^{a_k}-1} \cdots  y_n^{p^{a_n}}
 z_1^{p^{b_1}}\cdots z_{m-1}^{p^{b_{m-1}}}[z_m,y_k]+\\
 &&+\alpha y_1^{p^{a_1}} \cdots  y_n^{p^{a_n}} z_1^{p^{b_1}}\cdots z_{m-1}^{p^{b_{m-1}}}z_m.
\end{eqnarray*}

By Lemma \ref{lemavariavcomutam} we have 
\[z_{m-1}^{p^{b_{m-1}}}[z_m,y_k]+I=z_{m-1}^{p^{b_{m-1}}-1}z_m[z_{m-1},y_k]+I=
-z_{m-1}^{p^{b_{m-1}}-1}[z_{m-1},y_k]z_m+I.\]
Thus $\overline{f}+I=\widetilde{f}z_m+I$ where 

\begin{eqnarray*}
\widetilde{f}&=&
 -\sum_{k=1}^n \alpha_k y_1^{p^{a_1}}\cdots y_k^{p^{a_k}-1} \cdots  y_n^{p^{a_n}}
 z_1^{p^{b_1}}\cdots z_{m-1}^{p^{b_{m-1}}-1}[z_{m-1},y_k]+\\
 &&+\alpha y_1^{p^{a_1}} \cdots  y_n^{p^{a_n}} z_1^{p^{b_1}}\cdots z_{m-1}^{p^{b_{m-1}}}.
\end{eqnarray*}
Since $\widetilde{f}z_m \in C$, by Proposition \ref{propositioncomznadireita} we have that
$\widetilde{f}z_m \in (I+ 
\langle z_1z_2 \rangle^{TS(*)}+\langle  y_1^p \rangle^{TS(*)})$. Therefore
\[f \in (I+ 
\langle z_1z_2 \rangle^{TS(*)}+\langle  y_1^p \rangle^{TS(*)}).\]

\

Case 5. $m\geq 1$ and $b_m \geq 1$.

By Lemma \ref{lemavariavcomutam},  
\[z_{m}^{p^{b_{m}}-1}[z_m,y_k]+I=-z_{m}^{p^{b_{m}}-2}[z_m,y_k]z_m+I.\]
Thus $\overline{f}+I=\widetilde{f}z_m+I$ where 
\begin{eqnarray*}
\widetilde{f}&=&
 -\sum_{k=1}^n \alpha_k y_1^{p^{a_1}}\cdots y_k^{p^{a_k}-1} \cdots  y_n^{p^{a_n}}
 z_1^{p^{b_1}}\cdots z_{m}^{p^{b_{m}}-2}[z_m,y_k]+\\
 &&+\alpha y_1^{p^{a_1}} \cdots  y_n^{p^{a_n}} z_1^{p^{b_1}}\cdots z_{m}^{p^{b_{m}}-1}.
\end{eqnarray*}
Since $\widetilde{f}z_m \in C$, by Proposition \ref{propositioncomznadireita} we have 
\[\widetilde{f}z_m \in (I+ 
\langle z_1z_2 \rangle^{TS(*)}+\langle  y_1^p \rangle^{TS(*)}).\]
Therefore
$f \in (I+ 
\langle z_1z_2 \rangle^{TS(*)}+\langle  y_1^p \rangle^{TS(*)}).$

By the five cases it follows that 
\[C \subseteq (I+ 
\langle z_1z_2 \rangle^{TS(*)}+\langle  y_1^p \rangle^{TS(*)}) \]
as desired. 
The proof is complete.
\end{proof}

\subsection{$C(UT_2(F),\star)$ when $F$ is a finite field}

Let $F$ be a finite field with $|F|=q$ elements and char$(F)=p \neq 2$. Since $( \ F-\{0\} \ , \ \cdot \ )$ is a group, we have $a^{q-1}=1$
for all $a\in F-\{0\}$. In particular,
\[a^q=a\]
for all $a\in F$. Hence, if 
$$Y=\begin{pmatrix}
a & b \\ 
0 & a
\end{pmatrix}$$
then 
\begin{equation}\label{potenciay}
Y^i=\begin{pmatrix}
a^i & ia^{i-1}b \\ 
0 & a^i
\end{pmatrix} \ , \ \ 
Y^q=\begin{pmatrix}
a^q & qa^{q-1}b \\ 
0 & a^q
\end{pmatrix}=
\begin{pmatrix}
a & 0 \\ 
0 & a
\end{pmatrix}
\end{equation}
for all $a,b \in F$ and $i\geq 1$.

The next lemma is direct consequence of \cite[Proposition 4.2.3]{drenskybook}. See \cite[Lemma 2.1]{ronalddimas} too.

\begin{lemma} \label{lemapolcomutident}
Let $F$ be a finite field with $|F|=q$. Let $f \in F \langle X \rangle$ be a polynomial given by
\[f(x_1,\ldots,x_n)=\sum_{d_1=0}^{q-1} \ldots \sum_{d_n=0}^{q-1}\alpha_{(d_1,\ldots,d_n)}x_1^{d_1}\cdots x_n^{d_n},\]
where $\alpha_{(d_1,\ldots,d_n)} \in F$. If $f$ is a polynomial identity for $F$ then 
$\alpha_{(d_1,\ldots,d_n)}=0$ for all $(d_1,\ldots,d_n)$.
\end{lemma}

Let $\Lambda_n$ be the set of all elements $(s_1,\ldots,s_n) \in \mathbb{Z}^n$ such that: 
\begin{itemize}
\item[a)] $0\leq s_1,\ldots,s_n	 < 2q$,
\item[b)]If $s_i\geq q$ for some $i$, then $s_j<q$ for all $j\neq i$. 
\end{itemize} 

The next theorem was proved in \cite[Theorem 5.9]{ronalddimas}.
\begin{theorem}   \label{teoremaprimeirainvolucao}
 Let $F$ be a finite field with $|F|=q$ elements and char$(F)\neq 2$. Consider the involution $\star$ defined
 in (\ref{definicaoprimeirainvolucao}). Denote by $J$ the $T(*)$-ideal generated by the polynomials
 \begin{eqnarray*}
[y_1,y_2], \ \ [z_1,z_2], \ \  [y_1,z_1][y_2,z_2], \ \ z_1y_1z_2-z_2y_1z_1,  \\
(y_1^q-y_1)[z_1,y_2], \ \ (y_1^q-y_1)(y_2^q-y_2), \ \ z_1^q-z_1, \\
(z_1^{q-1}-1)[z_1,y_1], \ \ (y_{1}^q-y_{1})z_{1}-2^{-1}[z_{1},y_{1}] .
\end{eqnarray*}
Then $Id(UT_2(F),\star)=J$. Moreover, the quotient vector space $F\langle Y\cup Z \rangle /J$ 
 has a basis consisting of all polynomials of the form
 \[
 \left\{
\begin{array}{ll}
y_1^{s_1}\cdots y_n^{s_n}z_1^{r_1}\cdots z_m^{r_m-1}[z_m,y_k]+J, & 0\leq s_1,\ldots,s_n,r_1,\ldots,r_m< q,\ \ r_m\geq 1,\\
&n\geq 1, \ m\geq 1, \  k\geq 1;\\
y_1^{s_1}\cdots y_n^{s_n}z_1^{r_1}\cdots z_m^{r_m}+J, & 0\leq s_1,\ldots,s_n,r_1,\ldots,r_m< q, \ \ r_m\geq 1, \\
&n\geq 1, \ m\geq 1;\\
y_1^{s_1}\cdots y_n^{s_n}+J, & (s_1,\ldots,s_n)\in \Lambda_n, \ \ n\geq 1.
\end{array}
\right.
.\] 
\end{theorem}

\begin{proposition}\label{polinomio central da finitud}
Let $F$ be a finite field with $|F| = q$ and char$(F) \neq 2$. 
Then 
\[l y_1 (y_2^{q+l-1} -  y_2^l) + y_1^q y_2^l\]
is a $*$-central polynomial for $ (UT_2(F),\star)$ for all $l\geq 0$.
\end{proposition}

\begin{proof}
Denote $f(y_1,y_2)=l y_1 (y_2^{q+l-1} -  y_2^l) + y_1^q y_2^l$ and consider  
\[Y_i = 
\begin{pmatrix}
a_i & b_i \\
0 & a_i
\end{pmatrix},\]
where $a_i,b_i \in F$, $i=1,2$. By (\ref{potenciay}), we have 
\[ Y_2^l = \begin{pmatrix}
a_2^l & la_2^{l-1}b_2 \\
0 & a_2^l
\end{pmatrix} \ \mbox{and} \
Y_2^{q + l - 1} = \begin{pmatrix}
a_2^l & (l-1)a_2^{l-1}b_2 \\
0 & a_2^l
\end{pmatrix}.\] Thus
\[ Y_2^{q+l-1} -  Y_2^l =
\begin{pmatrix}
0 & - a_2^{l-1}b_2 \\
0 & 0
\end{pmatrix} \ \mbox{and} \ 
l Y_1 (Y_2^{q+l-1} -  Y_2^l) =
\begin{pmatrix}
0 & - l a_1 a_2^{l-1}b_2 \\
0 & 0
\end{pmatrix}.\]
By (\ref{potenciay}), we have 
\[ Y_1^q Y_2^l = 
\begin{pmatrix}
a_1 & 0 \\
0 & a_1
\end{pmatrix} 
\begin{pmatrix}
a_2^l & l a_2^{l-1} b_2 \\
0 & a_2^l
\end{pmatrix} =
\begin{pmatrix}
a_1 a_2^l & l a_1 a_2^{l-1} b_2 \\
0 & a_1 a_2^l
\end{pmatrix}.\] 
Therefore 
\[f(Y_1,Y_2)= 
 \begin{pmatrix}
a_1a_2^l & 0 \\
0 & a_1a_2^l
\end{pmatrix} \in Z(UT_2(F))
\]
as desired. The proof is complete.
\end{proof}

\

Denote by $V$ the following $T(*)$-space:
\begin{equation}
V = \left\langle l y_1 (y_2^{q+l-1} -  y_2^l) + y_1^q y_2^l: \ l \geq 0 \right\rangle ^{TS(*)} .
\end{equation}

By Proposition \ref{polinomio central da finitud}, we obtain  
\[V+Id(UT_2(F),\star) \subseteq C(UT_2(F),\star).\]

Since char$(F)=p$, if $l=kp$ and $y_1=1$ then $
l y_1 (y_2^{q+l-1} -  y_2^l) + y_1^q y_2^l= y_2^{kp}.$
Hence 
\begin{equation} \label{potenciadeyemv}
 y_2^{kp} \in V
\end{equation}
for all $k\geq 0$.

From now on we write
\[f \equiv g \Longleftrightarrow f+V+Id(UT_2(F),\star)=g+V+Id(UT_2(F),\star).\]

\begin{lemma}\label{lema general de l}
Let $F$ be a finite field with $|F| = q$ and char$(F) \neq 2$. 
If $l, n\geq 0$ then 
\[ 
\left( l y_1\cdots y_n (y_{n+1}^{q+l-1} - y_{n+1}^l) +  y_1^q \cdots y_n^q y_{n+1}^l \right) \equiv 0 .\]
\end{lemma}
\begin{proof}
The case $n=0$ is consequence of Proposition \ref{polinomio central da finitud}. In fact, substituting in
\[l y_1 (y_2^{q+l-1} -  y_2^l) + y_1^q y_2^l\]
the variable $y_1$ by $1$, we will have $l(y_2^{q+l-1} -  y_2^l) + y_2^l \equiv 0$.

Suppose  $n\geq 1$. Denote $u =(1/2)(y_1 y_2\cdots y_n + y_n \cdots y_2 y_1)$ and $J=Id(UT_2(F),\star)$. 
Since $u$ is a symmetric polynomial it follows that 
\begin{equation}\label{aaa1}
v= lu(y_{n+1}^{q+l-1} - y_{n+1}^l) +  u^q  y_{n+1}^l \in  V. 
\end{equation}
Since $[y_i,y_j] \in J$ (see Theorem \ref{teoremaprimeirainvolucao}), we have 
\[y_iy_j+J= y_jy_i+J.\]  
Thus $u + J = y_1\cdots y_n + J$ and
$u^q + J =  y_1^q \cdots y_n^q + J.$
Hence 
\[v+J= l y_1\cdots y_n (y_{n+1}^{q+l-1} - y_{n+1}^l) + y_1^q \cdots y_n^q y_{n+1}^l + J . \]
Now we use (\ref{aaa1}) to finish the proof.
\end{proof}

\begin{corollary}\label{diminui potencia de y}
Let $F$ be a finite field with $|F| = q$ and char$(F)=p \neq 2$. 
If $ f(y_1,\ldots,y_n) \in F \left\langle  Y \cup Z \right\rangle $ and $p \nmid l$ then there exists 
$g(y_1,\ldots,y_n) \in F \left\langle  Y \cup Z \right\rangle$ such that
\[ f y_{n+1}^{q+l-1} \equiv  g y_{n+1}^l .\]  
\end{corollary}
\begin{proof}
By Lemma \ref{lema general de l}, we have
\[ 
y_1\cdots y_n y_{n+1}^{q+l-1}\equiv 
(y_1\cdots y_n -(l^{-1}) y_1^q \cdots y_n^q )y_{n+1}^l.\]
Since $f y_{n+1}^{q+l-1}$ is a linear combination of polynomials
\[y_{i_1}\cdots y_{i_m} y_{n+1}^{q+l-1},\]
we finish the proof.
\end{proof} 

\begin{proposition}\label{uma variavel y}
Let $F$ be a finite field with $|F| = q$ and char$(F)=p \neq 2$. 
Consider
\[f = \sum_{i=0}^{2q-1} \alpha_i y_1^i,\]
where $\alpha_i \in F$ for all $i$. If $f \in C(UT_2(F), \star)$ then 
$ f \equiv 0 $.
\end{proposition}
\begin{proof}
Let $g$ be given by 
\[g = \sum_{i=q}^{2q-1}  \alpha_i y_1^i = \sum_{i=1}^{q} \alpha_{q+i-1} y_1^{q+i-1} 
= \sum_{\substack{i=1 \\ p|i}}^q \alpha_{q+i-1} y_1^{q+i-1}
+ \sum_{\substack{i=1 \\ p\nmid i}}^q \alpha_{q+i-1} y_1^{q+i-1}. \] 
By Corollary \ref{diminui potencia de y}, we have 
\[ g \equiv \sum_{\substack{i=1 \\ p|i}}^q \alpha_{q+i-1} y_1^{q+i-1} + 
\sum_{\substack{i=1 \\ p\nmid i}}^q \beta_i y_1^i, \]
for some $\beta_i \in F$.
Thus, there exist $ \gamma_i \in F$ such that
\[ f \equiv \sum_{\substack{i=1 \\ p|i}}^q \alpha_{q+i-1} y_1^{q+i-1} + 
\sum_{i=0}^{q-1} \gamma_i y_1^i. \] 
By (\ref{potenciadeyemv}), we obtain
\begin{equation} \label{aaa2}
f \equiv  
\underbrace{\sum_{\substack{i=1 \\ p|i}}^q \alpha_{q+i-1} y_1^{q+i-1}
 + \sum_{\substack{i=1 \\ p \nmid i}}^{q-1} \gamma_i y_1^i}_h. 
 \end{equation}
It follows from Proposition \ref{polinomio central da finitud} that
\begin{equation} \label{aaa3}
C(UT_2(F),\star)\supseteq V+Id(UT_2(F),\star).
\end{equation}
Since $f\in C(UT_2(F), \star ) $,  by (\ref{aaa2}) and (\ref{aaa3}) we have that 
$h\in C(UT_2(F), \star ) $ where
 \begin{equation*}\label{equacao1 polinomio central da finitud}
h = \sum_{\substack{i=1 \\ p|i}}^q \alpha_{q+i-1} y_1^{q+i-1} + 
\sum_{\substack{i=1 \\ p \nmid i}}^{q-1} \gamma_i y_1^i.
 \end{equation*} 
If $a\in F$ and 
\[ Y = 
 \begin{pmatrix}
 a & 1 \\
 0 & a
 \end{pmatrix},\]
 then by (\ref{potenciay}), 
 \[h(Y)=
  \begin{pmatrix}
  h(a)& \left[\displaystyle  \sum_{\substack{i=1 \\ p|i}}^q \alpha_{q+i-1}(-1) a^{i-1}+\sum_{\substack{i=1 \\ p \nmid i}}^{q-1} \gamma_i ia^{i-1} \right]\\
 0 & h(a)
 \end{pmatrix}
 \]
Since  $ h(Y) \in Z(UT_2(F)) $  it follows that 
 \[\displaystyle  \sum_{\substack{i=1 \\ p|i}}^q \alpha_{q+i-1}(-1) a^{i-1}+\sum_{\substack{i=1 \\ p \nmid i}}^{q-1} \gamma_i ia^{i-1} = 0\]
for all $a\in F$. By Lemma \ref{lemapolcomutident}, we have
\[
 \left\{
\begin{array}{lll}
 \alpha_{q+i-1}(-1)=0,&1\leq i \leq q,& p\mid i \ ; \\
 \gamma_i i=0,&1\leq i \leq q-1,& p \nmid i.
\end{array}
\right. 
\]
Thus $h=0$ and $f \equiv 0$ as desired. 
\end{proof}

\begin{lemma}\label{potencia pq}
Let $F$ be a finite field with $|F| = q$ and char$(F)=p \neq 2$. 
Then $ y_1^{pq} - y_1^p \in Id(UT_2(F),\star) $.
\end{lemma}
\begin{proof}
If $a,b \in F$ and $ Y = 
\begin{pmatrix}
a & b \\
0 & a
\end{pmatrix},$
then by (\ref{potenciay}),
\[ Y^{pq} =(Y^q)^p= 
\begin{pmatrix}
a^p & 0 \\
0 & a^p
\end{pmatrix} = Y^p. \] 
Thus $ Y^{pq} - Y^p=0$ as desired.
\end{proof}

\begin{lemma}\label{adicionando ps} 
Let $F$ be a finite field with $|F| = q$ and char$(F)=p \neq 2$. 
If $ i\geq 0 $ then
\[\left(  i y_1 (y_2^{q+i-1} - y_2^i)  + y_1^q y_2^i\right)  y_3^p \equiv 0. \]
\end{lemma}

\begin{proof}
Denote $ u = (1/2) (y_1 y_3^p + y_3^p y_1)$ and $J=Id(UT_2(F),\star)$. Since $ u $ is a symmetric polynomial, we have
\begin{equation}\label{adicionando ps ecuacao}
i u (y_2^{q+i-1} - y_2^i) + u^q y_2^i \in V.
\end{equation}
Since $y_i y_j + J = y_j y_i +J$ it follows that 
\begin{equation}\label{aaa4}
 u + J = y_1 y_3^p + J.
\end{equation}
Thus, by Lemma \ref{potencia pq}, 
\begin{equation}\label{aaa5}
 u^q + J = y_1^q y_3^{pq} + J = y_1^q y_3^p + J.
\end{equation}
 Now we use (\ref{aaa4}) and (\ref{aaa5}) to obtain  
\begin{align}
 i u (y_2^{q+i-1} - y_2^i) + u^q y_2^i + J &= i y_1 y_3^p (y_2^{q+i-1} - y_2^i) + y_1^q y_3^p y_2^i + J  \nonumber \\
 &= i y_1 (y_2^{q+i-1} - y_2^i) y_3^p + y_1^q y_2^i y_3^p + J  \nonumber  \\
 &=\left(  i y_1 (y_2^{q+i-1} - y_2^i)  + y_1^q y_2^i\right)  y_3^p+J. \label{aaa6}
\end{align}
By (\ref{adicionando ps ecuacao}) and (\ref{aaa6}) it follows that
$\left(  i y_1 (y_2^{q+i-1} - y_2^i)  + y_1^q y_2^i\right)  y_3^p \equiv 0$ as desired.
\end{proof}

\begin{corollary}\label{f com y p}
Let $F$ be a finite field with $|F| = q$ and char$(F)=p \neq 2$. 
If $ f(y_1,\ldots,y_n) \equiv 0$ then $f(y_1,\ldots,y_n) y_{n+1}^{lp} \equiv 0$ for all $l\geq 0$.
\end{corollary}
\begin{proof}
Write $f=f_V+f_I$, where $f_V \in V$ and $f_I \in Id(UT_2(F),\star)$. Since 
\[fy_{n+1}^{lp}=f_Vy_{n+1}^{lp}+f_Iy_{n+1}^{lp} \equiv f_Vy_{n+1}^{lp}\]
we can suppose $f(y_1,\ldots,y_n) \in V $. In this case, $f$ is a linear combination of polynomials
\[ i g_1 (g_2^{q+i-1} - g_2^i) + g_1^q g_2^i, \]  
where $ g_1 , g_2 \in F \left\langle Y \cup Z \right\rangle ^+  $ and  $ i\geq 0 $. 
By Lemma \ref{adicionando ps} we have
\[ \left( i g_1 (g_2^{q+i-1} - g_2^i) + g_1^q g_2^i \right) y_{n+1}^p \equiv 0.\]  
Thus $f(y_1,\ldots,y_n) y_{n+1}^p \equiv 0$. Since $V+Id(UT_2(F),\star)$ is a $T(*)$-space we have
$f(y_1,\ldots,y_n) y_{n+1}^{lp}=f(y_1,\ldots,y_n) (y_{n+1}^{l})^p \equiv 0$ for all $l\geq 0$.
\end{proof}

\begin{proposition}\label{com variavel y}
Let $F$ be a finite field with $|F| = q$ and char$(F)=p \neq 2$.
If $f(y_1,\ldots,y_n) \in C(UT_2(F),\star) $ then $ f(y_1,\ldots,y_n)\equiv 0$.
\end{proposition}

\begin{proof}

By Theorem \ref{teoremaprimeirainvolucao} we can suppose
\[ f = \sum_{s \in \Lambda_n} \alpha_s y_1^{s_1} \cdots y_n^{s_n}, \]
where $ s=(s_1,\ldots,s_n)$, $ \alpha_s \in F $. 

We will prove the proposition using induction on $n$. 

The case $n=1$ is consequence of Proposition \ref{uma variavel y}.

Suppose $n\geq 2$. Write 
$$f = \sum_{i=0}^{2q-1} f_i y_n^i.$$
Note that:

a) If $0 \leq i \leq q-1$ then
\begin{align*}
f_i = \sum_{(s_1,\ldots,s_{n-1}) \in \Lambda_{n-1}} \alpha_{(s_1,\ldots,s_{n-1},i)} y_1^{s_1} \cdots y_{n-1}^{s_{n-1}}.
\end{align*}

b) If $ q \leq i \leq 2q-1$ then
\begin{align}\label{tipo maior que q}
f_i = \sum_{s_1,\ldots,s_{n-1}=0}^{q-1} \alpha_{(s_1,\ldots,s_{n-1},i)} y_1^{s_1} \cdots y_{n-1}^{s_{n-1}}.
\end{align}

Write
\[ g = \sum_{i=q}^{2q-1}  f_i y_n^i = \sum_{i=1}^{q} f_{q+i-1} y_n^{q+i-1} 
= \sum_{\substack{i=1 \\ p|i}}^q f_{q+i-1} y_n^{q+i-1}
+ \sum_{\substack{i=1 \\ p\nmid i}}^q f_{q+i-1} y_n^{q+i-1}. \] 
By Corollary \ref{diminui potencia de y}, there exist polynomials $g_i(y_1,\ldots,y_{n-1})$ such that
\[ g\equiv \sum_{\substack{i=1 \\ p|i}}^q f_{q+i-1} y_n^{q+i-1} + 
\sum_{\substack{i=1 \\ p\nmid i}}^q g_i y_n^i. \]
Thus, there exist polynomials $ h_i(y_1,\ldots,y_{n-1})$ such that
\begin{equation}\label{aaa7} f \equiv \underbrace{\sum_{\substack{i=1 \\ p|i}}^q f_{q+i-1} y_n^{q+i-1} + 
\sum_{i=0}^{q-1} h_i y_n^i}_h. 
\end{equation}
Denote \begin{equation*}
h(y_1,\ldots,y_n) = \sum_{\substack{i=1 \\ p|i}}^q f_{q+i-1} y_n^{q+i-1}
+ \sum_{i=0}^{q-1} h_i y_n^i.
\end{equation*}
Since $f\in C(UT_2(F,\star))$ and $V+Id(UT_2(F,\star)) \subseteq C(UT_2(F,\star))$, we have by (\ref{aaa7}) that 
$ h\in C(UT_2(F,\star)) $. 
 
Consider the following matrices 
\[Y_n = 
\begin{pmatrix}
a_n & 1 \\
0 & a_n
\end{pmatrix} \ \ \mbox{and} \ \ 
Y_k=\begin{pmatrix}
a_k & 0 \\
0 & a_k
\end{pmatrix}  \ (k\neq n)\]
where $a_j \in F$ for all $j$.
By (\ref{potenciay}) we have
\[Y_n^i = 
\begin{pmatrix}
a_n^i & i a_n^{i-1} \\
0 & a_n^i
\end{pmatrix}.\]
In particular, if $p|i$ then
\[Y_n^{q+i-1} = 
\begin{pmatrix}
a_n^i & - a_n^{i-1} \\
0 & a_n^i
\end{pmatrix}.\]
Hence 
\[h(Y_1,\ldots,Y_n)=\begin{pmatrix}
h(a_1,\ldots,a_n) & \overline{h}(a_1,\ldots,a_n) \\
0 & h(a_1,\ldots,a_n)
\end{pmatrix},\]
where
\begin{eqnarray*}
\overline{h}(a_1,\ldots,a_n) &=&
\sum_{\substack{i=1 \\ p|i}}^q f_{q+i-1}(a_1,\ldots,a_{n-1})(-1) a_n^{i-1}
+ \sum_{i=0}^{q-1} h_i(a_1,\ldots,a_{n-1}) i a_n^{i-1}\\
&=&\sum_{\substack{i=1 \\ p|i}}^q f_{q+i-1}(a_1,\ldots,a_{n-1})(-1) a_n^{i-1}
+ \sum_{\substack{i=1 \\ p \nmid i}}^{q-1} h_i(a_1,\ldots,a_{n-1}) i a_n^{i-1}.
\end{eqnarray*}
Since $h(Y_1,\ldots,Y_n) \in Z(UT_2(F)) $ we obtain 
$\overline{h}(a_1,\ldots,a_n)=0$ for all $a_1,\ldots,a_n \in F$.
By Lemma \ref{lemapolcomutident} it follows that 
\begin{equation}\label{equacao f segunda}
f_{q+i-1} (a_1,\ldots,a_{n-1}) = 0
\end{equation}
for all $i=1,\ldots, q$ where $p|i$ ; and 
\begin{equation}\label{equacao h segunda}
h_i (a_1,\ldots,a_{n-1}) = 0
\end{equation}
for all $i=1, \ldots,  q-1$ where $p \nmid i$.

By (\ref{tipo maior que q}), (\ref{equacao f segunda}) and Lemma \ref{lemapolcomutident}, we have 
$f_{q+i-1}(y_1,\ldots ,y_{n-1}) = 0 $ for all $i=1,\ldots, q$ where $p|i$.
Thus
\begin{equation} \label{aaa8}
h(y_1,\ldots,y_n) = \sum_{i=0}^{q-1} h_i y_n^i. 
\end{equation}
Since $h\in C(UT_2(F),\star)$, by Proposition \ref{propositconseq} it follows that 
$h_i y_n^i \in C(UT_2(F),\star)$ for all $i=0,\ldots, q-1$.  
Substituting in $h_iy_n^i$ the variable $y_n$ by $1$, it follows that 
\[h_i \in C(UT_2(F),\star)\]
for all $i=0,\ldots, q-1$. 
We have two cases:

\

a) Case $p \nmid i$. 

Consider $$Y_k = \begin{pmatrix}
a_k & b_k \\
0 & a_k
\end{pmatrix},$$
where $a_k,b_k \in F$ and $k=1,\ldots,n-1$. We have
\[ h_i(Y_1,\ldots,Y_{n-1}) = \begin{pmatrix}
h_i(a_1,\ldots,a_{n-1}) & \beta \\
0 & h_i(a_1,\ldots,a_{n-1})
\end{pmatrix} \]
for some $\beta \in F$. Since $h_i \in C(UT_2(F,\star))$ it follows that $\beta=0$. By (\ref{equacao h segunda}), 
we have $h_i(a_1,\ldots,a_{n-1})=0$ too. Thus 
$h_i(Y_1,\ldots,Y_{n-1})=0$ and $h_i(y_1,\ldots,y_{n-1}) \in Id(UT_2(F),\star)$ 
for all $i=0, \ldots,  q-1$ where $p \nmid i$.

\

b) Case $p | i$.

Since $h_i(y_1,\ldots,y_{n-1}) \in C(UT_2(F),\star)$, we obtain, by induction, that $h_i \equiv 0$. Thus, by 
Corollary \ref{f com y p}, it follows that $h_i y_n^i \equiv 0$ for all $i=0, \ldots,  q-1$ where $p | i$.

By (\ref{aaa7}), (\ref{aaa8}) and two cases above, we have
\[f\equiv h \equiv 0,\]
as desired. 
 \end{proof}

 \begin{theorem} \label{teorinvstar3}
Let $F$ be a finite field with $|F| = q$ and char$(F)=p \neq 2$.
Then 
\[C(UT_2(F),\star) = \left\langle l y_1 (y_2^{q+l-1} -  y_2^l) + y_1^q y_2^l: \ 1 \leq l \leq p \right\rangle ^{TS(*)} + \left\langle z_1 z_2 \right\rangle ^{TS(*)} +
Id(UT_2(F),\star).\]
\end{theorem}

\begin{proof}
Firstly, we will prove the following claim :

\

\noindent {\bf Claim 1.} The set $C(UT_2(F),\star)$ equals 
\[C(UT_2(F),\star) = \left\langle l y_1 (y_2^{q+l-1} -  y_2^l) + y_1^q y_2^l: \ l \geq 0 \right\rangle ^{TS(*)} + \left\langle z_1 z_2 \right\rangle ^{TS(*)} +
Id(UT_2(F),\star).\]

\noindent \emph{Proof of Claim 1.} Denote $J=Id(UT_2(F),\star)$ and 
\[V=\left\langle l y_1 (y_2^{q+l-1} -  y_2^l) + y_1^q y_2^l: \ l \geq 0 \right\rangle ^{TS(*)}.\]
Since $z_1z_2 \in C(UT_2(F),\star)$, we have $ C(UT_2(F),\star) \supseteq \left\langle z_1 z_2 \right\rangle ^{TS(*)} $.
Therefore, by Proposition \ref{polinomio central da finitud}, we obtain 
\begin{equation}\label{aaa9}
C(UT_2(F),\star) \supseteq \left(V + \left\langle z_1 z_2 \right\rangle ^{TS(*)} +J\right).
\end{equation}

Consider $f \in C(UT_2(F),\star)$. We will prove that $f\in \left(V + \left\langle z_1 z_2 \right\rangle ^{TS(*)} +J\right)$.
By Theorem \ref{teoremaprimeirainvolucao}, $f=f_J+f_{\Upsilon}$ where $f_J \in J$ and $f_{\Upsilon}$ 
is a linear combination
of polynomials
\[
 \left\{
\begin{array}{ll}
y_1^{s_1}\cdots y_n^{s_n}z_1^{r_1}\cdots z_m^{r_m-1}[z_m,y_k], & 0\leq s_1,\ldots,s_n,r_1,\ldots,r_m< q,\ \ r_m\geq 1,\\
&n\geq 1, \ m\geq 1, \  k\geq 1; \hspace{\fill} (\Upsilon_1)\\
y_1^{s_1}\cdots y_n^{s_n}z_1^{r_1}\cdots z_m^{r_m}, & 0\leq s_1,\ldots,s_n,r_1,\ldots,r_m< q, \ \ r_m\geq 1, \\
&n\geq 1, \ m\geq 1; \hspace{\fill} (\Upsilon_2)\\
y_1^{s_1}\cdots y_n^{s_n}, & (s_1,\ldots,s_n)\in \Lambda_n, \ \ n\geq 1. \hspace{\fill} (\Upsilon_3)
\end{array}
\right.
\] 
Since $f_J \in J$, we have 
\[f\in \left(V + \left\langle z_1 z_2 \right\rangle ^{TS(*)} +J\right) \Leftrightarrow
 f_{\Upsilon} \in \left(V + \left\langle z_1 z_2 \right\rangle ^{TS(*)} +J\right).
\]
Thus we can suppose $f=f_{\Upsilon}$. Since $\deg_{z_i}f <q$,
we can suppose $f$ a homogeneous polynomial in the variable $z_i$ for all $i=1,\ldots,m$
(see Proposition \ref{propositconseq}).
Denote $\deg_{z_i}f =r_i$.

If $r_1=\ldots=r_m=0$ then by Proposition \ref{com variavel y} 
\[f \in \left(V + \left\langle z_1 z_2 \right\rangle ^{TS(*)} +J\right).\]

Suppose $r_i \neq 0$ for some $i$. Renumbering the indices if necessary, we may assume that
$r_i \geq 1$ for all $i=1,\ldots,m$.
Since $f$ is a linear combination of polynomials in $\Upsilon_1$ and $\Upsilon_2$, we have $f=f_1+f_2$ where
\[
 f_1=f_1(y_1,y_2,\ldots,z_1,z_2,\ldots) = \sum_{n,k,s} \alpha_{(n,k,s)}y_1^{s_1}\cdots y_n^{s_n}z_1^{r_1}\cdots z_m^{r_m-1}[z_m,y_k]
\]
with $0\leq s_1,\ldots,s_n< q$, \ \ $1\leq r_1,\ldots,r_m< q$, 
$\ \ n\geq 1, \ \ m\geq 1, \  \ k\geq 1$,
\ \ $s=(s_1,\ldots,s_n)$, \ \ $\alpha_{(n,k,s)}\in F$; and 
\[
 f_2=f_2(y_1,y_2,\ldots,z_1,z_2,\ldots) = \sum_{n,s} \beta_{(n,s)}y_1^{s_1}\cdots y_n^{s_n}z_1^{r_1}\cdots z_m^{r_m}
\]
with $0\leq s_1,\ldots,s_n< q$, \  \ $1\leq r_1,\ldots,r_m< q$, $ \ \ n\geq 1, \ \ m\geq 1$,
\ \ $s=(s_1,\ldots,s_n),$ \  \  $\beta_{(n,s)}\in F$.

We have three cases:

\

Case 1. $m=1$ and $r_m=1$.

In this case, $f=f_1+f_2$ where
\[
 f_1=f_1(y_1,y_2,\ldots,z_1) = \sum_{n,k,s} \alpha_{(n,k,s)}y_1^{s_1}\cdots y_n^{s_n}[z_1,y_k]
\]
and 
\[
 f_2=f_2(y_1,y_2,\ldots,z_1) = \sum_{n,s} \beta_{(n,s)}y_1^{s_1}\cdots y_n^{s_n}z_1.
\]
Let
\[
 Y_i=\left(
 \begin{array}{cc}
 a_i&0  \\
 0&a_i 
 \end{array}
\right) \  \  \mbox{and} \ \ 
Z_i=\left(
 \begin{array}{cc}
 1&0  \\
 0&-1 
 \end{array}
\right)
 \]
where $a_i \in F$. 
We have
\[f(Y_1,Y_2,\ldots,Z_1)=
 \left(\begin{array}{cc}
       \displaystyle \sum_{n,s} \beta_{(n,s)}a_1^{s_1}\cdots a_n^{s_n} &\theta \\
       0&\displaystyle -\sum_{n,s} \beta_{(n,s)}a_1^{s_1}\cdots a_n^{s_n}
       \end{array}
\right)
\]
where $\theta \in F$. Since $f(Y_1,Y_2,\ldots,Z_1) \in Z(UT_2(F))$, it follows that $\theta =0$ and 
\[\sum_{n,s} \beta_{(n,s)}a_1^{s_1}\cdots a_n^{s_n}=0\]
for all $a_1,\ldots, a_n \in F$. Since $0\leq s_1,\ldots,s_n< q$  we have, by Lemma \ref{lemapolcomutident}, that
$\beta_{(n,s)}=0$ for all $n,s$. Thus $f=f_1$. If $\overline{Y}_1,\overline{Y}_2, \ldots \in UT_2(F)^+$ and
$\overline{Z}_1 \in UT_2(F)^-$ then 
\[f(\overline{Y}_1,\overline{Y}_2,\ldots,\overline{Z}_1)=
f_1(\overline{Y}_1,\overline{Y}_2,\ldots,\overline{Z}_1)=\alpha e_{12} \]
for some $\alpha \in F$. Since $f(\overline{Y}_1,\overline{Y}_2,\ldots,\overline{Z}_1)\in Z(UT_2(F))$, we obtain
$\alpha=0$, that is $f \in J$. Therefore $f \in \left(V + \left\langle z_1 z_2 \right\rangle ^{TS(*)} +J\right)$.

\

Case 2. $m\geq 2$ and $r_m=1$.

In this case, $f=f_1+f_2$ where
\[
 f_1=f_1(y_1,y_2,\ldots,z_1,z_2,\ldots) = 
 \sum_{n,k,s} \alpha_{(n,k,s)}y_1^{s_1}\cdots y_n^{s_n}z_1^{r_1}\cdots z_{m-1}^{r_{m-1}}[z_m,y_k]
\]
and 
\[
 f_2=f_2(y_1,y_2,\ldots,z_1,z_2,\ldots) = 
 \sum_{n,s} \beta_{(n,s)}y_1^{s_1}\cdots y_n^{s_n}z_1^{r_1}\cdots z_{m-1}^{r_{m-1}}z_m.
\]
By Lemma \ref{lemavariavcomutam}, we have 
\[z_{m-1}^{r_{m-1}}[z_m,y_k]+J=z_{m-1}^{r_{m-1}-1}z_m[z_{m-1},y_k]+J=
-z_{m-1}^{r_{m-1}-1}[z_{m-1},y_k]z_m+J.\]
Thus $f+J=\widetilde{f}z_m+J$ where 

\begin{eqnarray*}
\widetilde{f}&=&
-\sum_{n,k,s} \alpha_{(n,k,s)}y_1^{s_1}\cdots y_n^{s_n}z_1^{r_1}\cdots z_{m-1}^{r_{m-1}-1}[z_{m-1},y_k]
+\\
 &&+\sum_{n,s} \beta_{(n,s)}y_1^{s_1}\cdots y_n^{s_n}z_1^{r_1}\cdots z_{m-1}^{r_{m-1}}.
\end{eqnarray*}
Since $\widetilde{f}z_m \in C(UT_2(F),\star)$ we have, by Proposition \ref{propositioncomznadireita}, that
$\widetilde{f}z_m \in ( 
\langle z_1z_2 \rangle^{TS(*)}+J)$. Therefore
\[f \in \left(V+ 
\langle z_1z_2 \rangle^{TS(*)}+J\right).\]

\

Case 3. $m\geq 1$ and $r_m\geq 2$.

By Lemma \ref{lemavariavcomutam}, we have 
\[z_{m}^{r_m-1}[z_m,y_k]+J=-z_{m}^{r_m-2}[z_m,y_k]z_m+J.\]
Thus $f+J=\widetilde{f}z_m+J$ where
\begin{eqnarray*}
\widetilde{f}&=&-\sum_{n,k,s} 
\alpha_{(n,k,s)}y_1^{s_1}\cdots y_n^{s_n}z_1^{r_1}\cdots z_m^{r_m-2}[z_m,y_k]\\
 &&+\sum_{n,s} \beta_{(n,s)}y_1^{s_1}\cdots y_n^{s_n}z_1^{r_1}\cdots z_m^{r_m-1}.
\end{eqnarray*}
Since $\widetilde{f}z_m \in C(UT_2(F),\star)$ we have, by Proposition \ref{propositioncomznadireita}, that
$\widetilde{f}z_m \in ( 
\langle z_1z_2 \rangle^{TS(*)}+J)$. Therefore
\[f \in \left(V+ 
\langle z_1z_2 \rangle^{TS(*)}+J\right).\]

\

We prove that 
\[C(UT_2(F),\star) \subseteq \left(V + \left\langle z_1 z_2 \right\rangle ^{TS(*)} +J\right).\]
By (\ref{aaa9}) we finish the proof of Claim 1.

\

\

\noindent {\bf Claim 2.} If $l\geq 0$ then
\[ l y_1 (y_2^{q+l-1} -  y_2^l) + y_1^q y_2^l \in \left\langle r y_1 (y_2^{q+r-1} -  y_2^r) + y_1^q y_2^r : 1\leq r \leq p \right\rangle ^{TS(*)} + J.\]
	
\

\noindent \emph{Proof of Claim 2.} If $l\geq 0$, let $k,r$ be the integers such that 
$l = k p + r$ and $0\leq r <p$. 
We have
\begin{equation}\label{abcdcba}
l y_1 (y_2^{q+l-1} -  y_2^l) + y_1^q y_2^l = r y_1 y_2^{k p} (y_2^{q+r-1} -  y_2^r) + y_1^q y_2^{k p} y_2^r. 
\end{equation}

\

Case 1. $r = 0$.

In this case, by (\ref{abcdcba}), it follows that
\[l y_1 (y_2^{q+l-1} -  y_2^l) + y_1^q y_2^l= y_1^q y_2^{kp} \in \left\langle y_1^q y_2^p \right\rangle ^{TS(*)}.\]
Note that
\[y_1^q y_2^p=p y_1 (y_2^{q+p-1} -  y_2^p) + y_1^q y_2^p.\]
This case is done.

\

Case 2. $1\leq r < p$.

Denote $u = (1/2) (y_1 y_2^{k p} + y_2^{k p} y_1)$. 
Since
\[y_1y_2+J=y_2y_1+J,\]
we have 
$u + J = y_1 y_2^{k p}+J$. Moreover, by Lemma \ref{potencia pq}, 
\[ u^q + J = y_1^q y_2^{k p q} + J = y_1^q y_2^{k p} + J. \]
Thus, by (\ref{abcdcba}), we obtain
\begin{eqnarray*}
 l y_1 (y_2^{q+l-1} -  y_2^l) + y_1^q y_2^l+J&=&r y_1 y_2^{k p} (y_2^{q+r-1} -  y_2^r) + y_1^q y_2^{k p} y_2^r + J \\
 &=&r u (y_2^{q+r-1} -  y_2^r) + u^q y_2^r + J.
\end{eqnarray*}
Therefore
\[l y_1 (y_2^{q+l-1} -  y_2^l) + y_1^q y_2^l \in 
 \left\langle r y_1 (y_2^{q+r-1} -  y_2^r) + y_1^q y_2^r \right\rangle^{TS(*)} + J\]
 as desired.
 
 \
 
 By the Claims 1 and 2 we complete the proof of theorem.
\end{proof}

\section{$*$-central polynomials for $(UT_2(F),s)$}

In this section, we describe the $*$-central polynomials for $(UT_2(F),s)$ where  
\begin{equation*}
  \left(
 \begin{array}{cc}
  a&c \\
  0&b
 \end{array}
\right)^{s}=\left(
 \begin{array}{cc}
  b&-c \\
  0&a
 \end{array}
\right) 
\end{equation*}
for all $a,b,c \in F$. Note that $\{e_{11}+e_{22}\}$ and $\{e_{11}-e_{22}, e_{12}\}$ form a basis for the
vector spaces $(UT_2(F))^{+}$ and $(UT_2(F))^{-}$ respectively.

The next theorem was proved in \cite[Theorem 3.2]{vinkossca}: 
 
\begin{theorem}\label{identidadesut2fsinfinito}
Let $F$ be an infinite field with $char(F)\neq 2$. Then $Id(UT_2(F),s)$ is the $T(*)$-ideal generated by 
the polynomials
\begin{equation*}
[y_1,y_2], \ \ [z_1,y_1], \ \  [z_1,z_2][z_3,z_4] \ \ \mbox{and} \ \  z_1z_2z_3-z_3z_2z_1.
\end{equation*}
\end{theorem}

\

The next theorem was proved in \cite[Theorem 6.15]{ronalddimas}:

\begin{theorem}\label{identidadesut2fsfinito}
 Let $F$ be a finite field with $|F|=q$ elements and char$(F)\neq 2$.  Then $Id(UT_2(F),s)$ is the $T(*)$-ideal generated by 
the polynomials
\begin{eqnarray*}
[y_1,y_2], \ \ [z_1,y_1], \ \  [z_1,z_2][z_3,z_4] \ \ \mbox{and} \ \  z_1z_2z_3-z_3z_2z_1,  \\
y_1^q-y_1, \ \ (z_1^q-z_1)(z_2^q-z_2), \ \ z_1^{q+1}-z_1^2,  \\
(z_{1}^q-z_{1})z_{2}+z_2(z_{1}^q-z_{1}), \ \  [z_{1}, z_{2}](z_3^{q}-z_3).
\end{eqnarray*}
\end{theorem}

\begin{theorem}\label{teoremainvs}
Let $F$ be a field of char$(F) \neq 2$. The set of all $*$-central polynomials of $(UT_2(F), s)$ is
\[C(UT_2(F),s)= Id(UT_2(F), s)+ 
\langle y_1 \rangle^{TS(*)}.\]
\end{theorem}

\begin{proof}
Since  $(UT_2(F))^+=Z(UT_2(F))$ we have 
\[C(UT_2(F),s)\supseteq Id(UT_2(F), s)+ 
\langle y_1 \rangle^{TS(*)}.\]
Let $f=f(y_1,\ldots,y_n,z_1,\ldots,z_m) \in C(UT_2(F),s)$ and write
 \[f=f^{+}+f^{-}\]
 where $f^{+}\in F\langle Y\cup Z \rangle^{+}$ and $f^{-}\in F\langle Y\cup Z \rangle^{-}$. If
$Y_1,\ldots, Y_n \in UT_2(F)^+$ and $Z_1,\ldots, Z_m \in UT_2(F)^-$ then
$f^-(Y_1,\ldots,Y_n,Z_1,\ldots,Z_m) \in UT_2(F)^-$ and 
\begin{multline*}
f(Y_1,\ldots,Y_n,Z_1,\ldots,Z_m)-f^+(Y_1,\ldots,Y_n,Z_1,\ldots,Z_m)=\\
f^-(Y_1,\ldots,Y_n,Z_1,\ldots,Z_m) \in UT_2(F)^+.
\end{multline*}
Thus $f^-(Y_1,\ldots,Y_n,Z_1,\ldots,Z_m)=0$ and $f^-\in Id(UT_2(F), s)$. Since
$f^+ \in \langle y_1 \rangle^{TS(*)}$ it follows that $f\in Id(UT_2(F), s)+ 
\langle y_1 \rangle^{TS(*)}$. We prove that 
\[C(UT_2(F),s)= Id(UT_2(F), s)+ 
\langle y_1 \rangle^{TS(*)}\]
as desired.
\end{proof}

\section{Conclusion}

Let $F$ be a field of char$(F)\neq 2$, and let $*$ be an involution of the first kind on $UT_2(F)$. 
By Corollary \ref{corolarioequivalenciainvolucoe},
we have
\[C(UT_2(F),*)=C(UT_2(F),\star) \ \ \mbox{or} \ \ C(UT_2(F),*)=C(UT_2(F),s).\]
Thus, by Theorem \ref{teorinvsta1}, Theorem \ref{teorinvsta2}, Theorem \ref{teorinvstar3} and Theorem \ref{teoremainvs}
we have described $C(UT_2(F),*)$.

Since 
\[z_1z_2 \notin Id(UT_2(F),\star) \ \ \mbox{and} \ \ y_1 \notin Id(UT_2(F),s),\]
it follows that 
\[C(UT_2(F),*)\neq Id(UT_2(F),*)+F.\]
Moreover, by Theorem \ref{teoremaidentidaprimeirainv}, Theorem \ref{teoremaprimeirainvolucao}, Theorem 
\ref{identidadesut2fsinfinito}, Theorem \ref{identidadesut2fsfinito} it follows that there exists a finite 
set $S$ such that
\[Id(UT_2(F),*)=\langle S \rangle^{T(*)}.\]
Since 
\begin{eqnarray*}
\langle S \rangle^{T(*)}=&\langle \ y_{n+1}fy_{n+2}, \ y_{n+1}fz_{m+1}, \
z_{m+1}fy_{n+1}, \ z_{m+1}fz_{m+2} \ \ | \\
&f=f(y_1,\ldots,y_n,z_1,\ldots,z_m)\in S  \ \ \mbox{or}
\ \ f^* \in S  \ \rangle^{TS(*)}, 
\end{eqnarray*}
we prove the Theorem \ref{teoremafinitogeradoresdeute}.

\section*{Acknowledgments}
 
The first author was supported by Ph.D. grant from CAPES. 
The second author was partially supported by FAPESP grant No. 2014/09310-5,
and by CNPq grant No. 406401/2016-0.


\begin{thebibliography}{99}

\bibitem{Aljgiakar}
 Eli Aljadeff, Antonio Giambruno and Yakov Karasik. 
 Polynomial identities with involution, superinvolutions and the Grassmann envelope.
 Proceedings of the American Mathematical Society 145 (2017) 1843-1857. 



\bibitem{bahturinbook}
Yu. A. Bahturin. Identical relations in Lie algebras. VNU Science Press (1987).

\bibitem{bekhocirrankin}
C. Bekh-Ochir and S.A. Rankin. 
The central polynomials of the infinite-dimensional unitary and nonunitary Grassmann algebras. 
Journal of Algebra and its Applications 09 (2010) 687-704.

\bibitem{belov}
A.Ya. Belov. On non-Specht varieties. Fundam. Prikl. Mat. 5 (1999) 47-66.


\bibitem{brandaoplamen}
Ant\^onio Pereira Brand\~ao Jr. and Plamen Koshlukov. Central polynomials for $\mathbb{Z}_2$-graded algebras
and for algebras with involution. Journal of Pure and Applied Algebra 208
(2007) 877-886.

\bibitem{brandaoplkrel}
Ant\^onio Pereira Brand\~ao Jr. , Plamen Koshlukov, Alexei Krasilnikov and \'Elida Alves da Silva. The Central
Polynomials for the Grassmann Algebra. Israel Journal of Mathematics
179 (2010) 127-144.


\bibitem{colombo}
Jones Colombo and Plamen Koshlukov. Central polynomials in the matrix algebra of
order two. Linear Algebra and its Applications 377 (2004), 53-67.


\bibitem{drenskybook}
Vesselin Drensky. Free algebras and PI-algebras: graduate course in
algebra. Springer-Verlag (2000).

\bibitem{drenskygiambruno}
Vesselin Drensky and Antonio Giambruno. Cocharacters, codimensions and Hilbert series of the polynomial identities for $2 \times 2$ matrices with involution.
 Canadian Journal of Mathematics 46 (1994) 718-733.


\bibitem{drenskyformanek}
Vesselin Drensky and Edward Formanek.
Polynomial Identity Rings. Birkh{\"a}user (2012).

\bibitem{drenskykasparia1}
Vesselin Drensky and Azniv Kasparian. Polynomial identities of eighth degree for $3 \times 3$ matrices,
Annuaire de l'Univ. de Sofia, Fac. de Math. et Mecan., Livre 1, Math. 77, (1983) 175-195.

\bibitem{drenskykasparia2}
Vesselin Drensky and Azniv Kasparian. A new central polynomial for $3 \times 3$ matrices. 
Communications in Algebra 13 (1985) 745-752.


\bibitem{formanek} Edward Formanek. Central polynomials for matrix rings. Journal of Algebra 23 (1972) 129-132.


\bibitem{grishin}
A.V. Grishin. Examples of T-spaces and T-ideals of characteristic $2$ without the finite basis property.
Fundam. Prikl. Mat. 5
(1999) 101-118.


\bibitem{kaplansky} Irving Kaplansky. Problems in the theory of rings. Report of a Conference on Linear Algebras, June 1956, in National Acad. of Sci.-National Research Council, Washington. Publ. 502 (1957) 1-3.  


\bibitem{kemer}
A.R. Kemer. Finite basability of identities of associative algebras. Algebra and Logic 26 (1987) 362–397.


\bibitem{maltsev}
Yu. N. Maltsev. A basis for the identities of the algebra of upper triangular matrices.
Algebra i Logika 10 (1971) 393-400 (in Russian). Translation: Algebra and Logic 10 (1971) 242-247.

\bibitem{okhitin}
S. Okhitin. Central polynomials of the algebra of second order matrices. Moscow Univ. Math. Bull.
43 (1988) 49-51.

\bibitem{razmyslov} Yu. P. Razmyslov. On a problem of Kaplansky (Russian) Izv. Akad. Nauk. SSSR. Ser. Math. 37 (1973) 483-501. Translation: Math USSR. Izv 7 (1973) 479-496.


\bibitem{rowen}
Louis Halle Rowen. Polynomial Identities
in Ring Theory. Academic Press (1980).

\bibitem{Shchigolev}
V.V. Shchigolev. Examples of infinitely based T-ideals. Fundam. Prikl. Mat. 5 (1999) 307-312.

\bibitem{shchigolev2}
V.V. Shchigolev. 
Finite basis property of T -spaces over fields of characteristic zero. Izv. Math. 65 (2001) 1041–1071.


\bibitem{siderov}
P. N. Siderov. A basis for identities of an algebra of triangular matrices over an arbitrary field.
Pliska Stud. Math. Bulgar. 2 (1981) 143-152 (in Russian).



\bibitem{diogo}
Diogo Diniz Pereira da Silva e Silva.
On the central polynomials with involution of $M_{1,1}(E)$.
Serdica Mathematical Journal 41 (2015) 277-292

\bibitem{specht}
W. Specht. Gesetze in Ringen. I. Math. Z. 52 (1950) 557-589.


\bibitem{sviridova}
I. Sviridova.
Finite basis problem for identities with involution. Preprint, arXiv: 1410.2233.


\bibitem{ronalddimas}
Ronald Ismael Quispe Urure and Dimas Jos\'e Gon\c{c}alves. Identities with involution for $2\times 2$ upper
triangular matrices algebra over a finite field. Linear Algebra and its Applications 544 (2018) 223-253.

\bibitem{vinkossca} Onofrio M. Di Vincenzo, Plamen Koshlukov and Roberto La Scala.
Involutions for upper triangular matrix algebras.
Advances in Applied Mathematics 37 (2006) 541-568.


\bibitem{vinko} 
Onofrio M. Di Vincenzo and Plamen Koshlukov. On the $*$-polynomial identities of
$M_{1,1}(E)$. Journal of Pure and Applied Algebra 215 (2011) 262-275.


\end{thebibliography}
 \end{document}